\documentclass[11pt, twoside, letterpaper]{article}
\usepackage{tikz}
\usepackage{verbatim}
\usepackage{color}
\usepackage{amsmath}
\usepackage{amssymb}
\usepackage{amsthm}
\usepackage{bbm}

\usepackage{url}

\newtheorem{theorem}{Theorem}
\newtheorem{lemma}{Lemma}

\newtheorem{corollary}{Corollary}

\newtheorem{question}{Question}

\newcommand{\E}{{{\bf E}}}

\usepackage[margin=1in]{geometry}

\usepackage[onehalfspacing]{setspace}

\title{Optimal promotions of new products on networks}

\author{Gadi Fibich and Amit Golan}
\date{\small{OR/MS Indices: Marketing (Promotion, New products), Networks/graphs\\
		Area of review: Stochastic Models}}

\makeatletter
\let\inserttitle\@title
\makeatother
\usepackage{fancyhdr}
\begin{document}
	\maketitle
	
	\begin{abstract}
		We present a novel methodology for analyzing the optimal promotion in the Bass model for the spreading of new products on networks. For general networks with $M$~nodes, the optimal promotion is the solution of $2^M-1$ nonlinearly-coupled boundary-value problems. On structured networks, however, the number of equations can be reduced to a managable size which is amendable to simulations and analysis. This enables us to gain insight into the effect of the network structure on optimal promotions. We find that the optimal advertising strategy decreases with time, whereas the optimal boosting of peer effects increases from zero  and then decreases. In low-degree networks, it is optimal to prioritize advertising over boosting peer effects, but this relation is flipped in high-degree networks. When the planning horizon is finite, the optimal promotion continues until the last minute, as opposed to an infinite planning horizon where the optimal promotion decays to zero. Finally, promotions with short planning horizons can yield an order of magnitude higher increase of profits, compared to those with long planning horizons.
	\end{abstract}
	
	\section{Introduction}
	Spreading processes on networks have attracted the attention of researchers  
	in mathematics, physics, biology, computer science, social sciences, economics, management science, and more,
	as they concern the spreading of ``items'' ranging from diseases and computer viruses to rumors, information, opinions, technologies, and 
	innovations~\cite{Albert-00,Anderson-92,Jackson-08,Pastor-Satorras-01,Strang-98}. The first quantitative model of the diffusion of new products was proposed in 1969 by 
	Bass~\cite{Bass-69}. In this model, individuals adopt an innovative product as a result of {\em external influences} by mass media and {\em internal influences} by individuals who already adopted the product (``word of mouth'', ``peer effect''). 
	
	The Bass model is a  compartmental model. Thus, the population is divided into two compartments, adopters and nonadopters, and individuals move between the two compartments at a rate which is given by an ordinary differential equation (ODE). The compartmental Bass model is one of the most cited papers in \textit{Management Science}~\cite{Hopp-04}. For many years, all of its extensions have also been compartmental models that are given by one or several deterministic ODEs. 
	Compartmental models are relatively easy to analyze. The main weakness of these models is that they implicitly assume that the underlying social network is an infinite complete graph~\cite{Fibichcompartmental,Niu-02}.  Real-life social networks, however, are sparse, as each individual only knows a tiny fraction of the population. Therefore, with the emergence of network science in the early 2000s, research has gradually shifted to the Bass model on networks, which is an agent-based model in which the adoption by each individual is stochastic, and it allows for any
	network structure.
	
	In order to accelerate the adoption process, a company can invest in a promotion that boosts the external and/or internal influences~\cite{Horsky-83, Muller2019, VanDenBulte2018, Schmitt2011}. This naturally led researchers to analyze the optimal promotion in the Bass model. Until now, these studies have only been conducted within the framework of compartmental Bass models~\cite{Bass1994,Fruchter-22,Fruchter2011, Horsky-83,Krishnan2006,Teng-1983,Teng-1984}. Most likely, this is because analyzing the optimal promotion in the Bass model on networks is much more challenging, as {\em there are no standard tools in optimal control theory for networks with stochastic dynamics}.
	
	As noted, compartmental models assume that the underlying social network is an infinite complete graph. Therefore, it has not not clear to what extent the theory of  optimal promotions on compartmental models remain valid for network models. In this work, we present the first analytical study of optimal time-varying promotional strategies in the Bass model on networks considering both uniform policies (i.e., identical across individuals) and targeted policies hat differentiate among them. We develop a systematic methodology that provides an exact system of $2^{M}-1$ coupled boundary-value problems for the optimal promotion on any network, where $M$ is the number of nodes. We then show that on structured networks, we can obtain a reduced system of ODEs, which is exact and amendable to analysis. This enables us to gain insight into how {\em the structure of the underlying network, as well as the duration and cost function of the promotion, affects the optimal promotion and the resulting profit gain}.
	
	\subsection{Literature review}
	\label{sec:Literature-review}
	
	In~\cite{Horsky-83}, Horsky et al.\ analyzed the optimal advertising strategy in the compartmental Bass model
	\begin{equation}
		\label{eq:Bass_intro}
		\frac{df}{dt}= \big(1-f\big)\Big(p(s(t))+qf\Big), \qquad f(0)=0,
	\end{equation}
	where $f(t)$ is the fraction of the population that adopted the product by time $t$, the parameters~$p$ and~$q$ account for the external and internal influence rates, and~$s(t)$ is the advertising spending rate. Here ``optimal'' refers to maximizing the accumulated profits
	\begin{equation*}
		\Pi[s(t)]:=\int_{t=0}^{T}e^{-\theta t}\Big(\gamma \frac{df}{dt}-s(t) \Big)\, dt,
	\end{equation*}
	where $T$ is the planning horizon, $\theta$ is the discount rate, and $\gamma$ is the income generated by the sale of one unit of the product.
	Horsky et al.\ showed that the optimal advertising strategy $s^{\rm opt}(t):={\rm argmax}_{s(t)} \Pi[s(t)]$ decreases with time.
	
	Subsequent studies showed that in the generalized Bass model, the optimal advertising strategy increases with time~\cite{Bass1994, Fruchter2011, Krishnan2006}. Fruchter and Van den Bulte~\cite{Fruchter2011} argued that these results run ``counter to empirical findings favoring high initial spending'', thus casting ``doubts on the value of the generalized Bass model for normative purposes''.
	Fruchter et al.\ \cite{Fruchter-22} analyzed a heterogeneous compartmental Bass model in which the population is divided into two groups, such that the adopters in the first group influence the nonadopters in the second to adopt, and the adopters in the second group influence the nonadopters in the first group not to adopt. They found conditions under which the optimal advertising strategy for the first group is U-shaped with respect to time.
	
	Teng and Thompson~\cite{Teng-1983, Teng-1984} considered promotions that simultaneously increase~$p$ and~$q$, i.e., where $p=p(s(t))$ and $q=q(s(t))$. They found that the optimal promotion~$s^{\rm opt}(t)$ is bang-bang: It is either identically zero or it alternates between a positive value and zero. When there are multiple competing products, they observed numerically that $s^{\rm opt}(t)$ is no longer bang-bang but rather behaves smoothly: It can be identically zero, decay monotonically, or initially increase and then decay. Dockner and Jorgenson~\cite{Dockner-1988} showed that if the promotion $s(t)$ affects both $p$ and $q$, the planning horizon $T$ is sufficiently small, and the promotion has a stronger effect on $q$ than on $p$, then $s^{\rm opt}(t)$ is increasing.
	
	As far as we know, optimal promotions in the Bass model on networks 
	has only been studied numerically by Goldenberg et al.~\cite{Goldenberg2001}. They observed  that in small-world networks, it is more beneficial to advertise heavily early on, and decrease the advertising level with time.

	\subsubsection{Optimal seeding}
	
	The above studies considered promotions that vary in time and are uniform in space, and studied the qualitative behavior of $s^{\rm opt}(t)$ over time (i.e., whether it is increasing or decreasing). A different line of research has been devoted to the problem of optimal seeding on networks. Most studies have used the SI model on networks~\cite{Akbarpour2018, Banerjee-13, Beaman2021}, which is the Bass model with no external influences ($p\equiv 0$). Typically, in optimal seeding, the product is given for free to a given number of individuals (``seeds'') at $t=0$, and the goal is to select the nodes that would maximize the number of adopters at the end of the planning horizon. There are notable differences between the problems of optimal seeding and time-varying optimal promotions:
	\begin{enumerate}
		\item 
		In optimal seeding, the promotion occurs at $t=0$ and is non-uniform in space. In optimal promotions, the promotion varies in time and is typically uniform in space.
		
		\item 
		In optimal seeding, the goal is typically to maximize the number of adopters at the end of the planning horizon. In optimal promotion studies, the time at which an individual adopts the product is also significant, due to the discount factor.
		\item
		In optimal seeding, the number of seeds is  fixed, and the objective is to find the best locations for the seeds. In optimal promotion studies, there is an added dimension of determining the optimal amount of money to invest. In analogy to optimal seeding, this would be equivalent to determining the optimal number of seeds which would maximize the profits.
	\end{enumerate}

	In the Bass model on networks, any node can also adopt due to external influences, and so using optimal seeding (i.e., a non-uniform strategy) to maximize the adoption level at the end of the planning horizon is less important. Indeed, Rossman and Fisher~\cite{Rossman2021} showed that even with extremely low values of $p$, seeding highly-connected nodes in the Bass model on networks has a negligible benefit, compared to random seeding.

	\subsection{Methodological contributions}
	\label{subsec:methods}
	
	As noted, the main challenge in developing a methodology for the analysis of optimal promotions in the Bass model on networks is that the spreading dynamics are stochastic. Furthermore, we want the analysis to be exact, and not based on some approximation (mean-field, closure, \dots) whose accuracy is not always clear. In order to do that, our starting point are the master equations, which are $2^M-1$ linearly-coupled deterministic initial-value problems, where $M$ is the number of nodes. The solution of the master equations gives the exact expected adoption level as a function of time for a general network~\cite{fibich2021diffusion}. We then combine the master equations with Pontryagin's maximum principle from optimal control theory~\cite{kamienschwartz} to obtain an exact system of $2^{M}-1$ nonlinearly-coupled boundary-value problems for the optimal promotion on a general network.
	
	In theory, this means that we can find the optimal promotion on any network. In practice, analyzing or even solving numerically a system of $2^M-1$~boundary-value problems is only possible for small $M$. On structured networks that possess some symmetries, however, the system size can be significantly reduced, without sacrificing the exactness of the equations. For example, on a complete network, we arrive at a system of $M$ boundary-value problems. Moreover, as $M\to\infty$, this system can be reduced to a single boundary-value problem, without making any approximation. A similar exact reduction is demonstrated for infinite one-dimensional networks and for infinite heterogeneous complete networks.
	
	In previous studies of optimal promotions in compartmental Bass models, the promotion either affected only $p$, or it simultaneously affected $p$ and $q$. Here we adopt a different approach and allow for dual independent promotional policies, $s_p(t)$ and $s_q(t)$, that lead to an increase only in~$p$ and only in~$q$, respectively. There are several reasons for this modeling choice:
	\begin{enumerate}
		\item 
		Some promotions only lead to an increase in~$p$ (e.g., advertisements) or only in~$q$ (e.g., referral fees).
		\item
		In our model formulation, a promotion only in~$p$ is a special case where $s_q \equiv 0$, and a promotion only in~$q$ is a special case where $s_p\equiv 0$.
		\item The separation of the optimal promotion into~$s_p^{\rm opt}(t)$ and~$s_q^{\rm opt}(t)$ allows us to study how their relative sizes depend on the network properties. This, in turn, enables us
		to understand whether the external or the internal influences are the dominant mechanism in optimal promotions that simultaneously affect~$p$ and~$q$.
	\end{enumerate}
	
	Obviously, some of the choices in our model formulation  may not be applicable to some new products: having two independent controls $s_p(t)$ and $s_q(t)$, the specific assumption on the dependence of the adoption rate on $s_p(t)$ and $s_q(t)$, see~\eqref{eq:p_q_s}, and the specific profit functional that we use, see~\eqref{eq:profit}. We stress, however, that {\em the methodology of combining the master equations with Pontryagin's Maximum Principle, and the subsequent exact reduction of the number of equations by utilizing network symmetries, can be easily extended to other models of promotions on networks}. Moreover, this methodology can be used to analyze optimal vaccination or quarantine policies in epidemiological models on networks~\cite{Balderrama2022-qa}. For example, in Section~\ref{sec:alternative}, we demonstrate that our methodology readily extends to {\em alternative profit functionals}. In particular, we consider settings in which the cost of promotion in~$p$ scales with the number of non-adopters, while the cost of promotion in~$q$ scales with the number of interactions between adopters and non-adopters. This formulation is especially well suited for analyzing targeted promotion strategies, such as personalized advertising and referral-based discounts.
	
	In Section~\ref{sec:Literature-review} we saw that prior research on optimal promotions has primarily focused on the qualitative behavior of the optimal promotion function $s^{\rm opt}(t)$.  This study introduces a new analytic dimension by examining the ``{\em effectiveness}'' of the optimal promotion, namely, the relative profit increase that it generates compared to the baseline scenario, i.e., 
	\begin{equation}
		\label{eq:delta_pi}
		\Delta\Pi^{\rm opt}_{\rm rel}:=\frac{\Pi\left[{\bf s}^{\rm opt}(t)\right]-\Pi^0}{\Pi^0}, \qquad \Pi^0:=\Pi[{\bf s}\equiv 0].
	\end{equation}
	We use this measure to analyze the dependence of the effectiveness of the optimal promotion on the network structure, on the network size~$M$, and on the planning horizon~$T$.
	
	We also formulate the problem of {\em feedback (closed-loop)} optimal promotions on networks (Section~\ref{sec:closed-loop}). This setting is non-standard: the control depends on the current state of the system, while the objective is to maximize expected profit. Consequently, the problem falls outside the scope of classical Hamilton–Jacobi–Bellman analysis.

	\subsection{Characteristics of optimal promotions}
	\label{sec:Emerging-picture}
	
	In this paper we analyze and compute numerically the optimal promotion in the Bass model on finite and infinite homogeneous complete networks, on infinite homogeneous one dimensional networks, and on infinite heterogeneous complete networks. Based on these results, we can formulate some characteristics of optimal promotions on networks. We stress, however, that while these characteristics seem to be intuitive, whether they hold for all networks is an open problem.
	
	\begin{enumerate}

		\item  {\em The optimal advertising strategy~$s_p^{\rm opt}(t)$  decreases with time.} 
		
		Indeed, it is more advantageous to advertise early on when there are more potential adopters, and when each promotion-added adopter has more time to influence the non-adopters. 
		
		\item  {\em  The optimal word-of-mouth promotion~$s_q^{\rm opt}(t)$ starts from zero and initially increases with time. If the planing horizon~$T$ is not ``short'', after some time $s_q^{\rm opt}(t)$ changes course and decreases with time.}
		
		The impact of a word-of-mouth promotion increases with the number of adopters (that can influence the nonadopters) and with the number of nonadopters (that can be influenced by the adopters).
		At $t=0$ there are no adopters, and so $s_q^{\rm opt}(0)=0$.
		Initially, as $t$~increases, the number of adopters increases and there are plentiful of nonadopters. Therefore, $s_q^{\rm opt}(t)$ increases. 
		If $T$ is not ``short'', however, after some time the decrease in the number of non-adopters outweights the increase in the number of adopters. Therefore, from that time and onward, $s_q^{\rm opt}(t)$ decreases.
		
		\item {\em When $T$ is infinite, $s_p^{\rm opt}(t)$ and $s_q^{\rm opt}(t)$ decrease to zero as $t\to\infty$. When $T$ is finite, $s_p^{\rm opt}(t)$ and $s_q^{\rm opt}(t)$ remain positive until the very end of the planning horizon.}
		
		This is because when $T=\infty$, the entire population adopts the product as $t\to \infty$, and so there are no non-adopters left to be influenced.  When $T<\infty$, there are still non-adopters present at $t=T$, and so it makes sense to promote until the last minute.

		\item {\em In low-degree networks
			(small complete networks, infinite lines), 
			$s_p^{\rm opt}(t)$ is larger than or of comparative size to $s_q^{\rm opt}(t)$   throughout the promotion. In high-degree networks 
			(large or infinite complete homogeneous and heterogeneous networks), 
			after a short initial stage $s_q^{\rm opt}(t)$ becomes significantly larger than $s_p^{\rm opt}(t)$ and remains so until the end of the planning horizon.}
		
		In low-degree networks, each adopter 
		can only influence (directly or indirectly) a few non-adopters,
		whereas  in high-degree networks, each adopter can influence numerous nonadopters.

		\item {\em  The qualitative dynamics of the total spending rate $s_p^{\rm opt}(t)+s_q^{\rm opt}(t)$  can exhibit two different behaviors:
			\begin{enumerate}
				\item $s_p^{\rm opt}(t)+s_q^{\rm opt}(t)$ decreases for all times.
				
				\item  $s_p^{\rm opt}(t)+s_q^{\rm opt}(t)$ initially decreases and then increases. If $T$ is not ``short'', there is a third region where $s_p^{\rm opt}(t)+s_q^{\rm opt}(t)$ decreases again.
				
			\end{enumerate} 
		}
		
		Since  $s_p^{\rm opt}(t)$  decreases throughout the promotion and~$s_q^{\rm opt}(t)$ initially increases and then decreases, the dynamics of the total spending rate depends on the relative sizes of $s_p^{\rm opt}(t)$ and $s_q^{\rm opt}(t)$. Thus, behavior~(a) occurs when $s_p^{\rm opt}(t)$ is dominant over $s_q^{\rm opt}(t)$ at all times, and behavior~(b) occurs when $s_p^{\rm opt}(t)$ is initially dominant over $s_q^{\rm opt}(t)$, but later $s_q^{\rm opt}(t)$ becomes dominant over $s_p^{\rm opt}(t)$.

		\item {\em The ``effectiveness'' of an optimal promotion with a short planning horizon can be an order-or-magnitude larger than of that with a long planning horizon.}
		
		Following Libai et al.~\cite{Libai-2013},  promotions can increase the profits by ``acceleration''
		(speeding up of adoptions that would have happened even without a promotion) and by ``expansion'' (when a promotion leads to adoptions by customers that otherwise would not have adopted the product). When $T$ is short, the promotion mainly leads to an expansion. In contrast, 
		when $T$ is long, most individuals will adopt by $t=T$ even without a promotion. Therefore, the promotion mainly leads to acceleration. Since {\em  an expansion is much more profitable than an acceleration},
		promotions over short planing horizon are considerably more ``effective'' than over long ones.
		
		\item {\em If in the absence of a promotion the adoption process in the network~${\cal N}^A$ is slower than in the network~${\cal N}^B$, the optimal promotion in~${\cal N}^A$ will be more effective than in~${\cal N}^B$.}
		
		Intuitively, if the network~${\cal N}^B$ is more efficient at spreading the new product, the promotion has to ``work harder'' to further speed up the adoption process.

		\item  {\em If the population is homogeneous, the optimal promotion 
			is to promote uniformly to all individuals, rather than to apply a different promotion to each half of the population.}
		
		This is because, generally speaking, heterogeneity slows down the adoption process~\cite{fibich2021diffusion,Fibichcompartmental}.

	\end{enumerate}
	
	As noted, these characteristics are based on analysis and simulations of the Bass model on several networks. Whether they also hold on other networks, or for Bass models on networks that use different assumptions from ours, is an open problem.  The methodology that is developed in this paper, however, is quite general, and can be used to study this open problem.

	\subsection{Effect of network structure on optimal promotions}
	
	Most of the above characteristics of optimal promotions are valid both for compartmental models and for network models. Characteristic~4, however, is a network effect, as there are no low-degree compartmental models. This characteristic shows that the relative sizes of~$s_p^{\rm opt}(t)$ and~$s_q^{\rm opt}(t)$ in low-degree networks is very different from that predicted by compartmental models. Further research is needed to determine the relative sizes of~$s_p^{\rm opt}(t)$ and~$s_q^{\rm opt}(t)$ on  real-life social networks, which are somewhere ``in-between'' these two artificial networks. 
	
	Networks exhibit structural properties beyond simple sparsity or density. In the Bass model on networks, diffusion is governed by local structural features (e.g., clustering and degree heterogeneity), rather than global characteristics (e.g., diameter)~\cite{OR-10}. Understanding how such local properties affect optimal promotion requires considering more general network structures. While the methodology developed in this paper applies to any structure, the associated boundary-value problem becomes tractable only when the full system of master equations can be reduced to a low-dimensional system. We leave this direction for future research.

	\subsection{Paper organization}
	
	The remainder of the paper is organized as follows. In Section~\ref{sec:model}, we introduce the Bass model on networks and the problem of optimal promotions on networks. In Section~\ref{sec:gen}, we combine the master equations for the Bass model on networks with Pontryagin's maximum principle from optimal control theory to obtain an exact system of $2^{M}-1$~boundary-value problems for the optimal promotion on a general network. In Section~\ref{sec:finite}, we consider complete network with $M$~nodes. By utilizing the network symmetry, we obtain an exact system of $M$ boundary-value problems for the optimal promotion. We use this system to study the effect of the network size on the optimal promotion, and observe that as $M$~increases, the relative size of~$s^{\rm opt}_q(t)$ increases, but the effectiveness of the optimal promotion decreases. When the complete network is infinite, we can further reduce the number of equations to a single boundary-value problem for the optimal promotion (Section~\ref{sec:infinite}). This enables us to analyze the optimal promotion at the beginning and at the end of the planning horizon, and to show that short promotions are considerably more effective than long ones. In Section~\ref{sec:line}, we consider a sparse infinite network-- the infinite line. We derive an exact system of two boundary-value problems for the optimal promotion on the infinite line, which turns out to be closer to that on a small complete network than to that on an infinite complete network. In Section~\ref{sec:heterogeneous}, we consider the role of heterogeneity, by analyzing the optimal promotion on infinite complete heterogeneous networks that consist of two equal-size homogeneous groups. We consider promotions that are uniform across all the population, as well as promotions where each group is targeted differently. In Section~\ref{sec:suboptimal}, we show that in a homogeneous infinite complete network, uniform promotional strategies dominate non-uniform policies. In Section~\ref{sec:Alternative-models}, we examine the robustness of the assumption that~$p$ and~$q$ depend on promotional spending through a square-root functional form by considering alternative specifications, including a $k$th-root form and a logarithmic form. In Section~\ref{sec:closed-loop}, we formulate and briefly discuss the problem of closed-loop (feedback) promotions on networks. Finally, in Section~\ref{sec:alternative}, we present an alternative profit formulation where the promotion costs depend on the number of adopters and nonadopters, and discuss how this change affects the optimal promotion. The appendices contain most of the proofs, as well as a description of the two numerical methods used to solve the boundary-value problems for the optimal promotion.

	\section{Model formulation}
	\label{sec:model}
	
	In this section we formulate the mathematical model for optimal promotions of new products on networks. We represent the population by a graph with the nodes $\mathcal{M}=\{1,\dots,M\}$. The state of each individual (node) is a random variable so that 
	$$
	X_j(t)=\begin{cases}
		1, & \ {\rm if}\ j\ {\rm adopts\ by\ time}\ t,\\
		0, & \ {\rm otherwise,}
	\end{cases}
	\qquad j \in\mathcal{M}.
	$$	
	When the product is first introduced into the market, there are no adopters, and so
	\begin{subequations}
		\label{eqs:Bass-model}
		\begin{equation}
			X_j(0)=0, \qquad j \in\mathcal{M}.
		\end{equation}
		So long that node~$j$ did not adopt the product, its adoption rate at time $t$ is
		\begin{equation*}
			\lambda_j(t)=p_j(t)+\sum_{k=1}^{M}q_{k\to j}(t) X_k(t),  
		\end{equation*}
		where $p_j$ is the rate of external influences on~$j$, 
		and $q_{k\to j}$ is the rate of internal influences by $k$ on~$j$,
		provided that $k$ is already an adopter. 
		Let the firm implements the promotion ${\bf s}(t):=\left(s_p(t), s_q(t)\right)$, such that $s_p(t)$ and $s_q(t)$ are the spending rates per consumer that lead to an increase in~$p_j$ (e.g., advertisements) and in~$q_{k\to j}$ (e.g., referral fees), respectively. Then the adoption rate of~$j$ is 
		\begin{equation}
			\label{eq:pjqj}
			\lambda_j(t)=p_j(s_p(t))+\sum_{k=1}^M q_{k\to j}(s_q(t))X_k(t), \qquad j\in \mathcal{M}.
		\end{equation}
		
		Once $j$ adopts the product, it remains an adopter at all later times.
		Therefore,  as $\Delta t\to 0$,
		\begin{equation}
			\mathbb{P}\Big(X_j(t+ \Delta t)=1 \mid \{X_1(t),\dots,X_M(t)\}\Big)
			=\begin{cases}
				1 & \ {\rm if}\ X_j(t)=1,\\
				\lambda_j(t) \Delta t, & \ {\rm if}\ X_j(t)=0,
			\end{cases}\qquad j\in\mathcal{M}.
		\end{equation}
	\end{subequations}
	Equations~\eqref{eqs:Bass-model} constitute {\bf the Bass model on a network} for the stochastic adoption dynamics of a new product. 
	
	The optimal promotion can be formulated as the optimization problem
	\begin{subequations}
		\label{eq:profit}
		\begin{equation}
			\label{eq:profit-a}
			{\bf s}^{\rm opt}(t)={\rm argmax}_{{\bf s}(t)\in {\cal S}^2}\Pi\left[{\bf s}(t)\right], \qquad \Pi\left[{\bf s}(t)\right]:=\int_{t=0}^{T}\pi({\bf s}(t))\, dt,
		\end{equation}
		where
		$\mathcal{S}$
		is the set of piecewise-continuous functions in $[0,T]\to R^+$, 
		$T$ is the planning horizon, $\Pi$~is the expected accumulated profits per consumer between $t=0$ and $t=T$, 
		\begin{equation}
			\label{eq:profit-b}
			\pi({\bf s}(t)):=e^{-\theta t}\Big(\gamma \frac{df}{dt}-\big(s_p(t)+s_q(t)\big) \Big) 
		\end{equation}
		is the discounted profit rate  per consumer, $\theta$~is the discount rate, $\gamma$~is the income generated by the sale of one unit of the product, and 
		\begin{equation}
			\label{eq:frac_adopters}
			f(t):=\frac{1}{M}\sum_{j=1}^{M}\E[X_j(t)]
		\end{equation} 
	\end{subequations}
	is the expected adoption level (fraction of adopters) in the Bass model~\eqref{eqs:Bass-model}.
	Hence, \eqref{eq:profit} is an {\em optimal control problem on a deterministic network with stochastic spreading dynamics}. 
	
	\section{Optimal promotions on general networks}
	\label{sec:gen}
	
	Consider the Bass model~\eqref{eqs:Bass-model} on a general network. Let $\emptyset\neq \Omega\subset\mathcal{M}$ be a subset of the nodes, and denote the event that at time $t$ all of the nodes in $\Omega$ are non-adopters, and the probability of this event by
	$$
	S_{\Omega}(t):=\{X_m(t)=0,\ m\in\Omega\},\qquad [S_{\Omega}](t):=\mathbbm{P}\left(S_{\Omega}(t)\right),
	$$
	respectively. Let us also introduce the notations
	\begin{equation}
		\label{eq:notations}
		p_\Omega:=\sum_{m\in\Omega}p_m,\qquad q_{k\to\Omega}:=\sum_{m\in\Omega}q_{k\to m},\qquad [S_j]:=[S_{\{j\}}],\qquad [S_{\Omega,k}]:=[S_{\Omega\cup \{k\}}].
	\end{equation}
	The optimal promotion in the Bass model on a general network is given as follows.
	\begin{theorem}
		\label{thm:gen}
		Consider the Bass model~\eqref{eqs:Bass-model}. For any promotion ${\bf s}(t):=\left(s_p(t),s_q(t)\right)\in{\mathcal S}^2$, let $\{[S_\Omega]\}_{\emptyset\neq \Omega \subset \mathcal{M}}$ and $\{\Psi_\Omega\}_{\emptyset\neq \Omega \subset \mathcal{M}}$ be the solutions to the coupled boundary-value problems
		\begin{subequations}
			\label{eqs:thm_gen}
			\begin{align}
				\label{eq:s_prime_general}
				\frac{d[S_{\Omega}]}{dt}&=-\Big(p_\Omega(s_p(t)) +\sum_{k\in\Omega^c}q_{k\to\Omega}(s_q(t))\Big)[S_{\Omega}]+\sum_{k\in\Omega^c}q_{k\to\Omega}(s_q(t))[S_{\Omega,k}], \qquad &[S_{\Omega}](0)=1,
				\\
				\label{eq:lambda_prime_general}
				\frac{d\Psi_{\Omega}}{dt}&=\Big(\Psi_\Omega(t)-\mathbbm{1}_{|\Omega|=1}\frac{\gamma}{M}e^{-\theta t}\Big)\Big(p_\Omega(s_p(t)) +\sum_{k\in\Omega^c}q_{k\to\Omega}(s_q(t))\Big)\nonumber \\
				&\quad-\sum_{m\in\Omega}\Big(\mathbbm{1}_{|\Omega|>1}\Psi_{\Omega\backslash\{m\}}(t)-\mathbbm{1}_{|\Omega|=2}\frac{\gamma}{M}e^{-\theta t}\Big)q_{m\to\Omega\backslash\{m\}}(s_q(t)),\qquad &\Psi_\Omega(T)=0.
			\end{align}
			Then the optimal promotion ${\bf s}^{\rm opt}:=(s_p^{\rm opt}, s_q^{\rm opt})$, see~\eqref{eq:profit}, satisfies the condition
			\begin{equation}
				\label{eq:c}
				\frac{\partial H}{\partial s_p}=\frac{\partial H}{\partial s_q}=0,\qquad
				H:=\Big(-\frac{\gamma}{M} \sum_{j=1}^{M}\frac{d[S_j]}{dt}-\big(s_p(t)+s_q(t)\big) \Big)e^{-\theta t}+\sum_{\emptyset\neq \Omega \subset \mathcal{M}}\Psi_\Omega(t)\frac{d[S_\Omega]}{dt}.
			\end{equation}
		\end{subequations}
	\end{theorem}
	\begin{proof}
		See appendix~\ref{app:thm:gen}.
	\end{proof}
	
	The optimal promotion at time~$t$ must balance its immediate contribution to the profit rate with its indirect effect on future adoption via peer influence. However, {\em as the end of the planning horizon approaches, the problem becomes effectively myopic}: with no remaining future to affect, the optimal promotion reduces to maximizing the instantaneous profit rate~$\pi$. Consequently, inducing adoption near the terminal time yields consumers who would otherwise not have adopted.
	\begin{corollary}
		\label{cor:end}
		Assume the conditions of Theorem~{\rm \ref{thm:gen}}. Let $T<\infty$. Then
		\begin{equation}
			\label{eq:argmaxTgen}
			{\bf s}^{\rm opt}(T)={\rm argmax}_{{\bf s}(T)} \pi(T),
		\end{equation}
		where $\pi$ is defined in~\eqref{eq:profit-b}.
	\end{corollary}
	\begin{proof}
		By~\eqref{eq:lambda_prime_general}, $\Psi_{\Omega}(T)=0$ for all $\Omega\in\mathcal{M}$. Hence, by~\eqref{eq:profit-b} and~\eqref{eq:c}, $H(T)=\pi(T)$, and thus ${\rm argmax}_{{\bf s}(T)} H(T)={\rm argmax}_{{\bf s}(T)} \pi(T)$.
	\end{proof}

	In theory, one could use Theorem~\ref{thm:gen} to compute the optimal promotion on any network. To do that, however, requires solving the nonlinearly-coupled system of the $2^M-1$
	boundary-value problems for $\{[S_{\Omega}],\Psi_\Omega\}$,  
	where $\Omega$ runs over  all  the nontrivial subsets of~$\mathcal{M}$. Therefore, in this manuscript we will compute the optimal promotions on networks that possess properties that enable us to reduce the number of equations to a managable size.

	\section{Optimal promotions on finite complete networks}
	\label{sec:finite}
	
	Consider a homogeneous complete network with $M$ nodes, where
	\begin{equation}
		\label{eq:complete}
		p_j(t)=p(t), \qquad q_{k\to j}(t)=\frac{q(t)}{M-1}\mathbbm{1}_{k\neq j},\qquad k,j\in\mathcal{M}.
	\end{equation}
	The adoption rate of node $j$ is
	\begin{equation*}
		\lambda_j^{\rm complete}=p(t)+\frac{q(t)}{M-1}N(t),
	\end{equation*}
	where $N(t):=\sum_{j=1}^{M}X_j(t)$ is the number of adopters at time $t$. Denote the expected adoption level on this network by $f^{\rm complete}(t; p(t), q(t),M)$. Because of the symmetry of the complete network~\eqref{eq:complete}, the non-adoption probability $[S_\Omega]$ only depends on the number of nodes in $\Omega$, and not on their identity.\,\footnote{e.g., $[S_{\{1,3\}}]=[S_{\{4,5\}}]$.} Therefore, we can denote 
	\begin{equation}
		\label{eq:comp_reduction}
		[S^n]:=\big[S_\Omega\ \Big|\ |\Omega|=n\big], \qquad n=1,\dots, M.
	\end{equation}
	This observation enables us to reduce the number of boundary-value problems in~\eqref{eqs:thm_gen} from~$2^M-1$  to~$M$:
	\begin{theorem}
		\label{lem:opt_M}
		Consider the Bass model~\eqref{eqs:Bass-model} on the homogeneous complete network~\eqref{eq:complete}. Let $\{[S^n](t)\}_{n=1}^M$ and $\{\Psi_n(t)\}_{n=1}^M$ be the solutions of the boundary-value problem 
		\begin{subequations}
			\label{eqs:pq_opt_M}
			\begin{align}
				\label{eq:S_M}
				\frac{d[S^n]}{dt}&=-\Big(np(s_p(t))+n\frac{M-n}{M-1}q(s_q(t))\Big)[S^n]+n\frac{M-n}{M-1}q(s_q(t))[S^{n+1}],\qquad &[S^n](0)=1,\\
				\label{eq:lambda_M}
				\frac{d\Psi_n}{dt}&=\Big(\Psi_n-\mathbbm{1}_{n=1}\gamma e^{-\theta t}\Big)\Big(np(s_p(t))+n\frac{M-n}{M-1}q(s_q(t))\Big) \\
				&\qquad\qquad-\left(\mathbbm{1}_{n>1}\Psi_{n-1}-\mathbbm{1}_{n=2}\gamma e^{-\theta t}\right)(n-1)\frac{M-n+1}{M-1}q(s_q(t)), \qquad &\Psi_n(T)=0.\nonumber
			\end{align}
		\end{subequations}
		Then the optimal promotion ${\bf s}^{\rm opt}:=(s_p^{\rm opt}, s_q^{\rm opt})$, see~\eqref{eq:profit}, satisfies the condition
		\begin{equation}
			\label{eq:H_comp}
			\frac{\partial H}{\partial s_p}\bigg|_{{\bf s}^{\rm opt}}=\frac{\partial H}{\partial s_q}\bigg|_{{\bf s}^{\rm opt}}=0,\qquad
			H:=\Big(-\gamma\frac{d[S]}{dt}-\big(s_p(t)+s_q(t)\big)\Big)e^{-\theta t}+\sum_{n=1}^{M}\Psi_n\frac{d[S^n]}{dt}.
		\end{equation}
	\end{theorem}
	\begin{proof}
		See Appendix~\ref{app:lem:opt_M}.
	\end{proof}
	
	In order to obtain explicit expressions for $s_p^{\rm opt}(t)$ and $s_q^{\rm opt}(t)$, we must specify functional forms for $p(s_p)$ and $q(s_q)$. It is natural to assume that these functions are smooth, monotonically-increasing, and concave. Accordingly, following~\cite{Fruchter-22}, we set
	\begin{equation}
		\label{eq:p_q_s}
		p(s_p):=p_0+b_p\sqrt{s_p}, \qquad q(s_q):=q_0+b_q\sqrt{s_q}, \qquad s_p,s_q\geq 0.
	\end{equation}
	In this formulation, $b_p\sqrt{s_p}$ captures the effect of promotional spending on the external influence rate, whereas $b_q\sqrt{s_q}$ reflects how such spending enhances either the receptivity of non-adopters to peer influence or the contagiousness of adopters. The robustness of the square-root specification for the dependence of~$p$ and~$q$ on promotional spending is examined in Section~\ref{sec:Alternative-models}.

	\begin{corollary}
		\label{cor:opt_complete_M}
		Assume the conditions of Theorem~{\rm \ref{lem:opt_M}}. Let the effect of the promotion be given by~\eqref{eq:p_q_s}. Then equation~\eqref{eq:H_comp} reduces to
		\begin{equation}
			\label{eq:s_lambda_pq_M}
			\begin{aligned}
				s^{\rm opt}_p(t)&=\frac{b_p^2}{4}\Big(\gamma[S]-e^{\theta t}\sum_{n=1}^{M}\Psi_n [S^n]\Big)^2,\\
				s^{\rm opt}_q(t)&=\frac{b_q^2}{4}\Big(\gamma\left([S]-[S^2]\right)-e^{\theta t}\sum_{n=1}^{M-1}\Psi_n n \frac{M-n}{M-1}\big([S^n]-[S^{n+1}]\big)\Big)^2.
			\end{aligned}
		\end{equation}
	\end{corollary}
	\begin{proof}
		See Appendix~\ref{app:cor:opt_complete_M}.
	\end{proof}
	
	Figure~\ref{fig:pq_opt_M2} presents the numerical solution of~(\ref{eqs:pq_opt_M},~\ref{eq:s_lambda_pq_M}) for a small population ($M=3$) and an infinite planning horizon $\left(T=\infty\right)$.\footnote{See Appendix~\ref{app:numerical} for the numerical methods used in this paper.} The ``external promotion'' $s_p^{\rm opt}(t)$ starts from a positive value and decreases. This is to be expected, because as more individuals adopt the product, there are less potential adopters, and so $s_p^{\rm opt}(t)$ becomes less effective. In contrast, the ``internal promotion'' $s_q^{\rm opt}(t)$ starts from zero, increases to a global maximum, and then decreases. 
	Intuitively, $s_q^{\rm opt}(0)=0$, since there are no adopters at $t=0$, and so a promotion in~$q$ does not boost peer effects. As the number of adopters increases, $s_q^{\rm opt}(t)$ becomes more effective in boosting adoptions, and so $s_q^{\rm opt}(t)$ increases. As the number of adopters further increases, however, there are fewer non-adopters, and so $s_q^{\rm opt}(t)$ becomes less effective. Therefore, $s_q^{\rm opt}(t)$ decreases. 
	
	The optimal promotion leads to a considerable increase of the adoption level, that is, $f^{\rm opt}(t)$ is considerably higher than $f^{0}(t):=f(t; {\bf s}\equiv 0)$, see Figure~\ref{fig:pq_opt_M2}B. This increase translates into a relative increase of the overall profits of $\Delta\Pi^{\rm opt}_{\rm rel} \approx 12\%$, see~\eqref{eq:delta_pi}. 
	
	As the population size is increased, the qualitative dynamics of $s_p(t)$ and of $s_q(t)$ remain unchanged, namely, $s_p(t)$ decreases and $s_q(t)$ increases from zero to a global maximum and then decreases~(Figure~\ref{fig:pq_opt_M}). However, the relative magnitudes of $s_p(t)$ and $s_q(t)$ change. Thus, when $M=3$, $s_q(t)$ remains smaller than $s_p(t)$ at all times, since there are few non-adopters for peer effects to have a significant impact. As $M$ increases, $s_q(t)$ becomes higher than $s_p(t)$ as soon as enough individuals adopt the product, as there are still numerous non-adopters for the  promotion in~$q$ to be effective. 
	
	The extra profits due to the optimal promotion {\em decrease} with~$M$, from $\Delta\Pi^{\rm opt}_{\rm rel}\approx 14\%$ for $M=2$ to $\Delta\Pi^{\rm opt}_{\rm rel}\approx 8.6\%$ for $M=100$ (Figure~\ref{fig:pq_opt_M}E). Since in the absence of promotions the expected adoption level $f^{\rm complete}(t; M)$ increases with $M$, see~\cite{Bass-monotone-convergence-23}, 
	this is a manifestation of the ``principle'' that promotions are less effective on ``faster'' networks (Section~\ref{sec:Emerging-picture}).

	\begin{figure}[ht!]
		\centering
		\scalebox{0.65}{\includegraphics{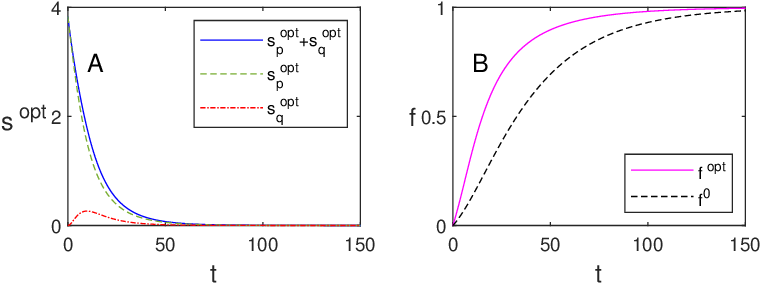}}
		\caption{A) The optimal promotion on a complete homogeneous network with $M=3$ nodes. The dash, dash-dot, and solid lines are $s_p^{\rm opt}(t)$, $s_q^{\rm opt}(t)$, and $s_p^{\rm opt}(t)+s_q^{\rm opt}(t)$, respectively. B) The expected adoption levels $f^{\rm opt}(t)$ in the presence of the optimal promotion (solid) and $f^{0}(t)$ in the absence of a promotion (dashes). Here $p_0=0.01$, $b_p=0.01$, $q_0=0.1$, $b_q=0.1$, $\gamma=1000$, $\theta=0.01$, and $T=\infty$. Here, $\Delta\Pi^{\rm opt}_{\rm rel}\approx 12\%$.} 
		\label{fig:pq_opt_M2}
	\end{figure}

	\begin{figure}[ht!]
		\centering
		\scalebox{0.65}{\includegraphics{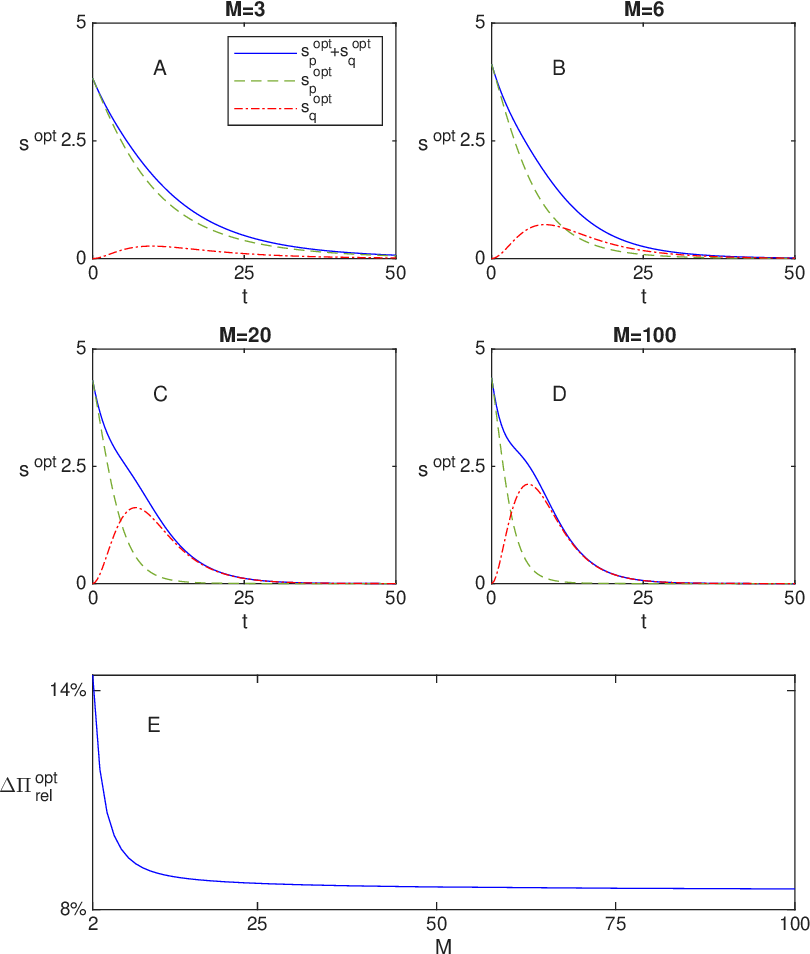}}
		\caption{A-D)~Same as Figure~\ref{fig:pq_opt_M2}A, for various values of $M$. E)~The dependence of $\Delta\Pi^{\rm opt}_{\rm rel}$ on $M$.}
		\label{fig:pq_opt_M}
	\end{figure}

	\section{Optimal promotions on infinite complete networks}
	\label{sec:infinite}
	
	In Section~\ref{sec:gen}, we showed that computing the optimal promotion on general networks is challenging. In Section~\ref{sec:finite}, we demonstrated that, for complete homogeneous networks with $M$~nodes, the optimal promotion is characterized by a system of $M$~coupled boundary-value problems. We now show that, for infinite complete networks, the problem can be reduced, without any approximation, to a single boundary-value problem involving only two equations. This reduction enables not only efficient numerical computation of the optimal strategy but also a deeper analytical understanding of its properties. To this end, we first recall the following convergence result.\footnote{Niu~\cite{Niu-02} established the convergence of~$f^{\rm complete}$ to $f^{\rm compart}$ under the assumption that the rate parameters~$p$ and~$q$ are constant over time. In the context of optimal promotions, however, it is necessary to extend this result to the case where these parameters vary over time~\cite{Bass-monotone-convergence-23}.}
	
	\begin{theorem}[\cite{Bass-monotone-convergence-23}]
		\label{thm:complete_convergence}
		Let $f^{\rm complete}(t;p(t),q(t),M)$ denote the expected adoption level in the Bass model~\eqref{eqs:Bass-model} on the complete network~\eqref{eq:complete} with time-dependent parameters
		$p(t)$ and~$q(t)$. Then $\lim\limits_{M\to\infty}f^{\rm complete}(t;p(t),q(t),M)=f^{\rm compart}(t; p(t),q(t))$, where $f^{\rm compart}(t; p(t),q(t))$ is the solution of the compartmental Bass model with time-dependent parameters
		\begin{equation}
			\label{eq:complete_limit}
			\frac{df}{dt}=(1-f)\Big(p(t)+q(t)f\Big), \qquad f(0)=0.
		\end{equation}
	\end{theorem}
	The reduction of the infinite system of master equations to a single ordinary differential equation enables us to express the optimal promotion as the solution of a single boundary-value problem:
	\begin{theorem}
		\label{lem:opt_complete_pq}
		Consider the Bass model~\eqref{eqs:Bass-model} on the homogeneous complete network~\eqref{eq:complete} as $M\to\infty$. Then the optimal promotion ${\bf s}^{\rm opt}:=(s_p^{\rm opt}, s_q^{\rm opt})$, see~\eqref{eq:profit}, satisfies the equation
		\begin{equation}
			\label{eq:H_inf}
			\frac{\partial H}{\partial s_p}=\frac{\partial H}{\partial s_q}=0, \qquad
			H:=\Big(\gamma \frac{df}{dt}-s_p(t)-s_q(t)\Big)e^{-\theta t}+\Psi\frac{df}{dt},
		\end{equation}
		where $f(t)$ and $\Psi(t)$ are the solutions of the boundary-value problem 
		\begin{equation}
			\label{eq:bvp_pq}
			\begin{aligned}
				\frac{df}{dt}&=(1-f)\Big(p(s_p(t))+q(s_q(t))f\Big),\qquad &f(0)=0,\\
				\frac{d\Psi}{dt}&=(\gamma e^{-\theta t}+\Psi)\Big(p(s_p(t))+q(s_q(t))(2f-1)\Big),\qquad &\Psi(T)=0.
			\end{aligned}
		\end{equation}
	\end{theorem}
	\begin{proof}
		See Appendix~\ref{app:lem:opt_complete_pq}.
	\end{proof}
	
		%
	
	\begin{corollary}
		\label{cor:opt_complete_pq}
		Assume the conditions of Theorem~{\rm \ref{lem:opt_complete_pq}}. Let the effect of the promotion be given by~\eqref{eq:p_q_s}. Then equation~\eqref{eq:H_inf} reduces to
		\begin{equation}
			\label{eq:s_lambda_pq}
			s^{\rm opt}_p(t)=\frac{b_p^2}{4}\left((1-f)(\Psi e^{\theta t}+\gamma)\right)^2 \!,\qquad
			s^{\rm opt}_q(t)=\frac{b_q^2}{b_p^2}f^2s^{\rm opt}_p(t).
		\end{equation}
	\end{corollary}
	\begin{proof}
		See Appendix~\ref{app:cor:opt_complete_pq}.
	\end{proof}
	As could be expected from Theorem~\ref{thm:complete_convergence}, the optimal promotion on infinite complete networks~(Figure~\ref{fig:pq_opt}A) is close to that on a complete network with $M=100$ nodes (Figure~\ref{fig:pq_opt_M}D).
	
	\begin{figure}[ht!]
		\centering
		\scalebox{0.65}{\includegraphics{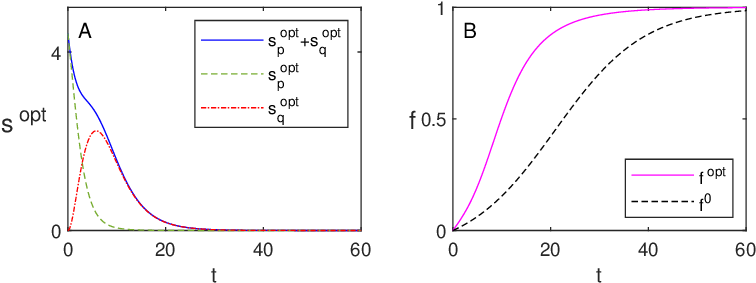}}
		\caption{Same as Figure~\ref{fig:pq_opt_M2} for an infinite complete network. Here, $\Delta\Pi^{\rm opt}_{\rm rel}=8.5\%$.}
		\label{fig:pq_opt}
	\end{figure}

	\subsection{Initial and final stages}
	\label{sec:init}
	
	We can analyze the behavior of $s_p^{\rm opt}(t)$ and $s_q^{\rm opt}(t)$ at the beginning of the promotion:
	\begin{corollary}
		\label{cor:sp_prime_zero}
		Assume the conditions of Corollary~{\rm \ref{cor:opt_complete_pq}}. Assume also that $q_0>\theta$. Then 
		$$
		s_p^{\rm opt}(0)>0, \qquad \frac{d}{dt}s_p^{\rm opt}(0)<0.
		$$
	\end{corollary}
	\begin{proof}
		See Appendix~\ref{app:cor_sp_prime_zero}.
	\end{proof}
	Therefore, $s_p^{\rm opt}(t)$ initially decreases, which is consistent with the discussion in Section~\ref{sec:Emerging-picture}.
	
	\begin{corollary}
		\label{cor:s(0)}
		Assume the conditions of Corollary~{\rm \ref{cor:sp_prime_zero}}. Then
		\begin{equation*}
			s_q^{\rm opt}(0)=0, \qquad \frac{d}{dt}{s_q^{\rm opt}}(0)=0, \qquad \frac{d^2}{dt^2}s_q^{\rm opt}(0)=2\frac{b_q^2}{b_p^2}s^{\rm opt}_p(0)>0.
		\end{equation*}
	\end{corollary}
	\begin{proof}
		Since $f(0)=0$, this follows from~\eqref{eq:s_lambda_pq} and Corollary~\ref{cor:sp_prime_zero}.
	\end{proof}
	Thus, $s_q^{\rm opt}(t)$ starts from zero and initially increases. Indeed, there is no point in enhancing peer effects at $t=0$, since there are no adopters. Furthermore, as the number of adopters initially increases, so does $s_q^{\rm opt}(t)$. Note that Corollary~\ref{cor:s(0)} shows that if $T$ is sufficiently small, $s_q^{\rm opt}(t)$ is increasing in $0\leq t\leq T$, which is consistent with~\cite{Dockner-1988}.

	Let us consider the adoption level and the spending rate at the end of the planning horizon:
	\begin{corollary}
		\label{cor:s_infty}
		Assume the conditions of Corollary~{\rm \ref{cor:opt_complete_pq}}.
		\begin{enumerate}
			\item 
			If $T<\infty$, then $f^{\rm opt}(T)<1$, $s_p^{\rm opt}(T)>0$, and $s_q^{\rm opt}(T)>0$.
			\item 
			If $T=\infty$, then $\lim\limits_{t\to\infty}f^{\rm opt}(t)=1$ and
			$
			\lim\limits_{t\to\infty}s_p^{\rm opt}(t)=\lim\limits_{t\to\infty}s_q^{\rm opt}(t)=0.
			$
		\end{enumerate}
	\end{corollary}
	\begin{proof}
		See appendix~\ref{app:cor:s_infty}.
	\end{proof}
	
	\begin{figure}[ht!]
		\centering
		\scalebox{0.65}{\includegraphics{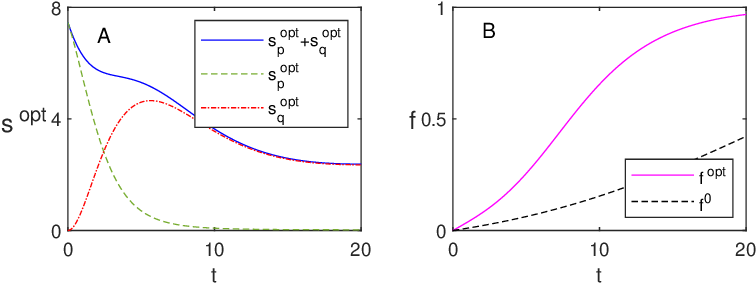}}
		\caption{Same as Figure~\ref{fig:pq_opt} for T=20. Here, $\Delta\Pi^{\rm opt}_{\rm rel}=118\%$.}
		\label{fig:T20_opt}
	\end{figure}
	
	Indeed, the optimal promotion at~$T$ is the maximizer of the profit rate at $T$, see~\eqref{eq:argmaxTgen}.
	Hence, $s_p^{\rm opt}(T)$ and $s_q^{\rm opt}(T)$  are positive if and only if there are still non-adopters at time~$T$. 
	When $T<\infty$, there exist nonadopters as $t\to T$ even in the presence of the optimal promotion, see Figure~\ref{fig:T20_opt}B. Therefore, the promotion continues until the last minute, see Figure~\ref{fig:T20_opt}A. In contrast, when $T$ is infinite there are no non-adopters  as $t\to\infty$, and so $s_p^{\rm opt}$ and $s_q^{\rm opt}$ vanish as $t\to T$, see Figure~\ref{fig:pq_opt}. 
	%
		
		\subsection{The ratio $s_q^{\rm opt}(t)/s_p^{\rm opt}(t)$}
		
		Since $s_p(0)>0$ and $s_q(0)=0$  (see Section~\ref{sec:init}), it follows that initially $s_p^{\rm opt}(t)>s_q^{\rm opt}(t)$. The following lemma determines whether and when this inequality reverses. 
		\begin{lemma}
			Assume the conditions of Corollary~{\rm \ref{cor:opt_complete_pq}}. 
			\begin{itemize}
				\item If $b_q<b_p$, then $s_p^{\rm opt}(t)>s_q^{\rm opt}(t)$ for all~$t$. 
				
				\item If $b_q>b_p$, then $s_p^{\rm opt}(t)>s_q^{\rm opt}(t)$ initially, and this inequality holds until the time at which $b_q f^{\rm opt}(t)=b_p$. Beyond this point, the inequality reverses, and $s_q^{\rm opt}(t)>s_p^{\rm opt}(t)$. 
			\end{itemize}
		\end{lemma}
		\begin{proof}
			The stated results follow directly from $\frac{s_q^{\rm opt}(t)}{s_p^{\rm opt}(t)}=\frac{b_q^2}{b_p^2}f^2(t)$, see~\eqref{eq:s_lambda_pq}.
		\end{proof}
		
		This lemma admits a natural interpretation. By~\eqref{eq:p_q_s} and~\eqref{eq:bvp_pq},
		$$
		\frac{df}{dt} =(1-f)\Big(p_0+q_0f\Big) +(1-f) \Big(b_p\sqrt{s_p} +b_q f \sqrt{s_q}\Big).
		$$	
		Consequently, when $b_qf<b_p$, the impact of~$s_p$ on the growth rate of~$f$ exceeds that of~$s_q$. Conversely, when $b_qf>b_p$, promotion in~$q$ becomes more effective.

		\subsection{Two scenarios for the total spending rate}

		Let us consider the total spending rate $s_p^{\rm opt}(t)+s_q^{\rm opt}(t)$ on infinite complete networks.
		\begin{corollary}
			Assume the conditions of Corollary~{\rm \ref{cor:sp_prime_zero}}. Then
			the total spending rate initially decreases.
		\end{corollary}
		\begin{proof}
			From Corollaries~\ref{cor:sp_prime_zero}~and~\ref{cor:s(0)} it follows that $\frac{d}{dt}\big(s_p^{\rm opt}+s_q^{\rm opt}\big)(0)<0$.
		\end{proof}
		This result is different from the one for 
		models where single promotion~$s^{\rm opt}(t)$ simultaneously influences~$p$ and~$q$,
		where $s^{\rm opt}(t)$ can be initially increasing or decreasing,
		depending on the parameters~\cite{Dockner-1988, Teng-1983, Teng-1984}.
		
		In Section~\ref{sec:init} we saw that $s_p^{\rm opt}(t)$ decreases throughout the promotion, whereas $s_q^{\rm opt}(t)$ increases from zero, and if $T$ is not ``short'', it later decreases. Therefore, depending on the relative magnitudes of $s_p^{\rm opt}(t)$ and $s_q^{\rm opt}(t)$, there are two possible scenarios: 
		\begin{enumerate}
			\item
			$s_p^{\rm opt}(t)+s_q^{\rm opt}(t)$ decreases throughout the promotion.
			\item
			$s_p^{\rm opt}(t)+s_q^{\rm opt}(t)$ initially decreases and  then  increases. If $T$ is not ``short'', there is a third region where  $s_p^{\rm opt}(t)+s_q^{\rm opt}(t)$ decreases again.
		\end{enumerate} 
		
		These two scenarios can be achieved by varying a single parameter. For example, in Figure~\ref{fig:p_effect} we vary the value of $p_0$, while keeping all other parameters fixed. When $p_0$ is small, $s_p^{\rm opt}(t)$ should be sufficiently large throughout the promotion, in order to compensate for the small value of~$p_0$. Therefore, $s_p^{\rm opt}(t)+s_q^{\rm opt}(t)$ behaves as~$s_p^{\rm opt}(t)$ and decreases. Conversely, for higher values of $p_0$, $s_p^{\rm opt}(t)$ can be small, as~$p_0$ is sufficiently high without any additional external promotion. Hence, initially $s_p^{\rm opt}(t)+s_q^{\rm opt}(t)$ behaves as $s_p^{\rm opt}(t)$ and decreases. As more individuals adopt the product, $s_q^{\rm opt}(t)$ becomes dominant over~$s_p^{\rm opt}(t)$, and so $s_p^{\rm opt}(t)+s_q^{\rm opt}(t)$ behaves as~$s_q^{\rm opt}(t)$, i.e., it increases and then decreases. 
		
		
		\begin{figure}[ht!]
			\centering
			\begin{minipage}{.28\textwidth}
				\centering
				\scalebox{0.65}{\includegraphics{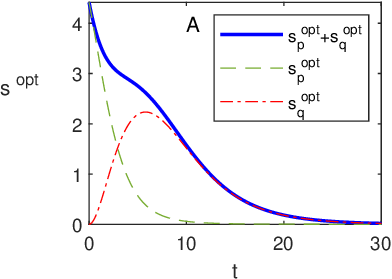}}
			\end{minipage}
			\begin{minipage}{.28\textwidth}
				\centering
				\scalebox{0.65}{\includegraphics{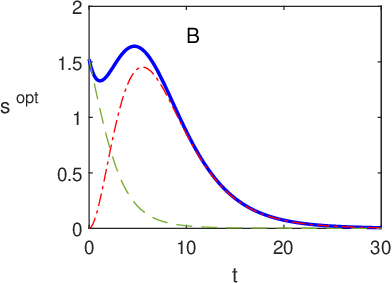}}
			\end{minipage}
			\caption{Same as Figure~\ref{fig:pq_opt}A for various values of $p_0$. A) $p_0=0.01$. B) $p_0=0.03$.}
			\label{fig:p_effect}
		\end{figure}

		\subsection{Effectiveness of optimal promotions}
		\label{subsec:eff}
		As noted in~\cite{Libai-2013}, a promotion can increase the profits by {\em acceleration} (i.e., by speeding up adoptions that would have happened even without the promotion) and by {\em expansion} (i.e., by leading to adoptions that would not have occurred otherwise by time $T$). When $T=\infty$, everyone adopts as $t\to T$ even without a promotion (since $\lim_{t\to \infty}f^0(t)=1$), and so a promotion ``only'' leads to acceleration. When $T<\infty$, however, $f^{0}(T)<1$. Hence, the promotion also leads to an expansion.

		
		Moving beyond the dichotomy of $T=\infty$ versus $T<\infty$ in Corollary~\ref{cor:s_infty}, we can ask:
		\begin{question}
			How does the effectiveness of the optimal promotion depend on the planing horizon~$T$?
		\end{question}
		
		The effectiveness of the optimal promotion is given by~$\Delta\Pi^{\rm opt}_{\rm rel}$; the relative increase in the overall profits due to the optimal promotion in $[0,T]$, see~\eqref{eq:delta_pi}.
		\begin{itemize}
			\item 
			When $T$ is short, the promotion mainly leads to an expansion. Therefore, as~$T$ increases,
			\begin{enumerate}
				\item 
				The promotion can influence more individuals.
				\item 
				Individuals who adopted because of the promotion have more time to lead to ``secondary adoptions'' through peer effects.
			\end{enumerate}
			Hence, $\Delta\Pi^{\rm opt}_{\rm rel}$ increases with $T$. 
			\item 
			When $T$ is long, most individuals will adopt by $t=T$ even without a promotion. Therefore, as~$T$ increases, the promotion leads to less expansion and more acceleration. Since an expansion is much more profitable than an acceleration, $\Delta\Pi^{\rm opt}_{\rm rel}$ decreases with $T$. 
		\end{itemize} 
		
		Indeed, Figure~\ref{fig:delta_pi} shows that $\Delta\Pi^{\rm opt}_{\rm rel}$ is increasing in $T$ for small $T$ and decreasing for large~$T$. The impact of the promotion on $\Delta\Pi^{\rm opt}_{\rm rel}$ over short planning horizons can be an order of magnitude larger than for long planning horizons. For example, in Figure~\ref{fig:delta_pi},~$\Delta\Pi^{\rm opt}_{\rm rel}$ can be above $800\%$ when $T$ is small, but is only about $8\%$ when $T$ is large. As noted, this is because having additional adoptions is much more profitable than having the adoptions occur earlier. 
		
		\begin{figure}[ht!]
			\centering
			\scalebox{0.5}{\includegraphics{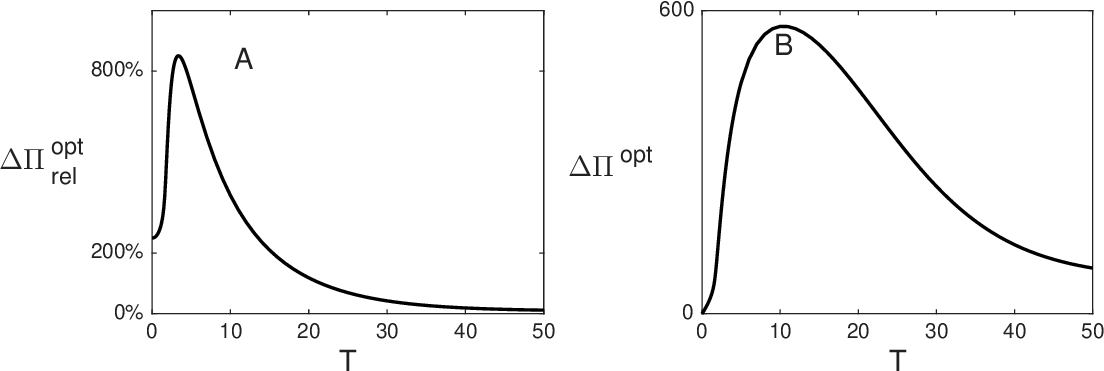}}
			\caption{The dependence of on $T$ of A)~$\Delta\Pi^{\rm opt}_{\rm rel}$, and B) $\Delta\Pi^{\rm opt}$. Parameters are as in Figure~\ref{fig:pq_opt}.}
			\label{fig:delta_pi}
		\end{figure}
		
		When selecting the optimal promotion time, a more relevant quantity for the firm is the absolute impact of the promotion, defined as $\Delta\Pi^{\rm opt}:=\Pi[{\bf s}^{\rm opt}]-\Pi^0$, where $\Pi^0:=\int_{t=0}^{T} e^{-\theta t}\gamma \frac{df^0}{dt}\, dt$ is the promotion-free profit. The maximum of $\Delta\Pi^{\rm opt}(T)$ occurs at a later time than that of~$\Delta\Pi^{\rm opt}_{\rm rel}(T)$, see Figure~\ref{fig:delta_pi}B, because the promotion-free profit increases with~$T$.

		\section{Optimal promotions on infinite lines}
		\label{sec:line}
		
		Consider now the ``opposite'' of a complete network, where each node can only be influenced by two nodes. Specifically, we consider the infinite homogeneous discrete line $\mathbb{Z}=\{\dots,-1,0,1,\dots\},$ where each node can be influenced by its two nearest neighbors. Thus, 
		\begin{equation}
			\label{eq:circle_conditions}
			p_j(t)=p(t),\qquad q_{k\to j}(t)=\frac{q(t)}{2}\mathbbm{1}_{|k-j|=1}, \qquad k,j\in\mathbb{Z}.
		\end{equation}
		Therefore, the adoption rate of node $j$ is
		$$
		\lambda_j^{\rm 1D}(t)= p(t)+\frac{q(t) }{2}\Big(X_{j-1}(t)+X_{j+1}(t)\Big), \qquad j\in\mathbb{Z}.
		$$
		
		The aggregate dynamics on the infinite line can also be reduced to a single ordinary differential equation:
		\begin{theorem}[\cite{Bass-monotone-convergence-23}]
			\label{thm:circle_convergence}
			Let $f^{\rm 1D}(t;p(t),q(t))$ denote the expected adoption level in the Bass model~\eqref{eqs:Bass-model} on the infinite line with time-dependent parameters~\eqref{eq:circle_conditions}. Then $f^{\rm 1D}(t)$ is the solution of the equation
			\begin{equation}
				\label{eq:circle_limit}
				\frac{df}{dt}=\left(p(t)+q(t)\Big(1-e^{-\int_{0}^{t}p}\Big)\right)(1-f), \qquad f(0)=0.
			\end{equation}
		\end{theorem}
		
		This exact reduction enables us to formulate the optimal promotion as a boundary-value problem with four equations:\footnote{There are four rather than two equations due to the introduction of the state variable $y$. This variable is needed, since the optimal promotion is found by computing $\frac{\partial H}{\partial s_p}$, and without $y$, $s_p$ would appear within an integral, see~\eqref{eq:circle_limit}.}
		\begin{theorem}
			\label{lem:opt_circle_pq}
			Consider the Bass model~\eqref{eqs:Bass-model} on the homogeneous infinite line~\eqref{eq:circle_conditions}. Let $f(t)$, $y(t)$, $\Psi_1(t)$, and $\Psi_2(t)$ be the solutions of the boundary-value problem 
			\begin{equation}
				\label{eq:bvp_pq_circle}
				\begin{aligned}
					\frac{df}{dt}&=\left(p(s_p(t))+q(s_q(t))(1-e^{-y})\right)(1-f), &f(0)=0,\\
					\frac{dy}{dt}&=p(s_p(t)),   &y(0)=0,\\
					\frac{d\Psi_1}{dt}
					&=(\gamma e^{-\theta t}+\Psi_1) \left(p(s_p(t))+q(s_q(t))(1-e^{-y})\right),  &\Psi_1(T)=0,\\
					\frac{d\Psi_2}{dt}
					&=-(\gamma e^{-\theta t}+\Psi_1)(1-f)q(s_q(t))e^{-y}, &\Psi_2(T)=0.
				\end{aligned}
			\end{equation}
			Then the optimal promotion ${\bf s}^{\rm opt}:=(s_p^{\rm opt}, s_q^{\rm opt})$, see~\eqref{eq:profit}, satisfies the condition
			\begin{equation}
				\label{eq:H_line}
				\frac{\partial H}{\partial s_p}=\frac{\partial H}{\partial s_q}=0, \qquad
				H:=\Big(\gamma\frac{df}{dt}-s_p(t)-s_q(t)\Big)e^{-\theta t}+\Psi_1\frac{df}{dt}+\Psi_2p(s_p(t)).
			\end{equation}
		\end{theorem}
		\begin{corollary}
			\label{cor:opt_circle_pq}
			Assume the conditions of Theorem~{\rm \ref{lem:opt_circle_pq}}. Let the effect of the promotion be given by~\eqref{eq:p_q_s}. Then equation~\eqref{eq:H_line} reduces to
			\begin{equation}
				\label{eq:s_lambda_pq_circle}
				s^{\rm opt}_p(t)=\frac{b_p^2}{4}\Big((\gamma +\Psi_1e^{\theta t})(1-f)+\Psi_2 e^{\theta t}\Big)^2,\quad
				s^{\rm opt}_q(t)=\frac{b_q^2}{4}\Big((\gamma +\Psi_1e^{\theta t})(1-f)(1-e^{-y})\Big)^2.
			\end{equation}
		\end{corollary}
		\begin{proof}
			See Appendix~\ref{app:lem:opt_circle_pq}.
		\end{proof}
		
		We can use~(\ref{eq:bvp_pq_circle},~\ref{eq:s_lambda_pq_circle}) to deduce some properties of ${\bf s}^{\rm opt}(t)$ on the infinite line, which mimic those on complete networks:
		\begin{corollary}
			\label{cor:s(0)_circle}
			Assume the conditions of Corollary~\ref{cor:opt_circle_pq}. Then $s_q^{\rm opt}(0)=0$.
		\end{corollary}
		\begin{proof}
			This follows directly from~\eqref{eq:s_lambda_pq_circle}.
		\end{proof}
		
		\begin{corollary}
			\label{cor:s_infty_circle}
			Assume the conditions of Corollary~\ref{cor:opt_circle_pq}.
			\begin{enumerate}
				\item 
				If $T<\infty$, then $f^{\rm opt}(T)<1$. Hence, $s_p^{\rm opt}(T)>0$ and $s_q^{\rm opt}(T)>0$.
				\item 
				If $T=\infty$, then $\lim\limits_{t\to\infty}f^{\rm opt}(t)=1$. Hence,
				$
				\lim\limits_{t\to\infty}s_p^{\rm opt}(t)=\lim\limits_{t\to\infty}s_q^{\rm opt}(t)=0.
				$
			\end{enumerate}
		\end{corollary}
		\begin{proof}
			See appendix~\ref{app:cor:s_infty_circle}.
		\end{proof}
		
		Similarly to a complete network, $s_p^{\rm opt}(t)$ decreases, whereas $s_q^{\rm opt}(t)$ increases from zero to a global maximum and then decreases (Figure~\ref{fig:pq_opt_circle}). The relative sizes of $s_p^{\rm opt}(t)$ and $s_q^{\rm opt}(t)$ are similar to those in a complete network with $M=6$ individuals~(Figure~\ref{fig:pq_opt_M}B), in the sense that $s_q^{\rm opt}(t)$ barely makes it above $s_p^{\rm opt}(t)$. In particular, on the infinite line, $s_q^{\rm opt}(t)$ does not reach as high of a value as on an infinite complete network with the same parameters~(Figure~\ref{fig:pq_opt}). Intuitively, this is because on the infinite line, an increase in~$q$ only impacts the few individuals that are sufficiently close to individuals which have already adopted. Hence, $s_q(t)$ has less of an impact on the adoption level.
		
		The relative increase of the profit on the infinite line (with an infinite planning horizon) due to the optimal promotion is $\Delta\Pi^{\rm opt}_{\rm rel}\approx 12\%$. This value is about $50\%$ higher than for the infinite complete network with the same parameters, see Figure~\ref{fig:pq_opt}. Since the expected adoption level in the infinite complete network is much higher than on the infinite line~\cite{OR-10}, this serves as another manifestation of the ``principle'' that promotions are more effective on ``slower'' networks (Section~\ref{sec:Emerging-picture}).
		
		\begin{figure}[ht!]
			\centering
			\scalebox{0.65}{\includegraphics{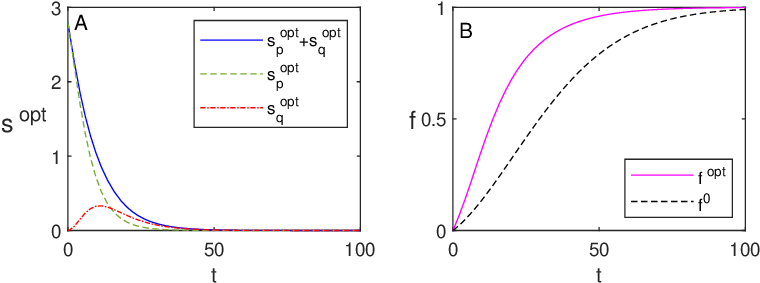}}
			\caption{Same as Figure~\ref{fig:pq_opt_M2}, on the infinite line. Here, $\Delta\Pi^{\rm opt}_{\rm rel}\approx 12\%$.}
			\label{fig:pq_opt_circle}
		\end{figure}

		\section{Optimal strategies on heterogeneous networks}
		\label{sec:heterogeneous}
		
		In order to allow for heterogeneity among consumers, we consider a complete heterogeneous network that consists of two equal-size groups denoted by $\mathcal{M}_1:=\{1,\dots,\frac{M}{2}\}$ and $\mathcal{M}_2:=\{\frac{M}{2}+1,\dots,M\}$, each of which is homogeneous. 	For any node in group~$\mathcal{M}_k$, the external and internal influence rates are~$p^k(t)$ and ~$q^k(t)$, i.e.,
		\begin{equation}
			\label{eq:het_conditions}
			p_{j}(t):=p^k(t),\qquad q_{m\to j}(t)=\frac{q^k(t)}{M-1}\mathbbm{1}_{m\neq j},\qquad j\in \mathcal{M}_k, \qquad m\in\mathcal{M}, \qquad k=1,2.
		\end{equation}
		Hence, the adoption rate of node $j$ is
		\begin{equation*}
			\label{eq:discrete_group_model}
			\lambda_j(t)=p^{k}(t)+\frac{q^k(t)}{M-1}N(t),\qquad j\in\mathcal{M}_k, \qquad k=1,2,
		\end{equation*}
		where $N(t)=\sum_{j=1}^{M}X_j(t)$ is the number of adopters at time~$t$. We also allow the two groups to differ in their response to the promotion, so that if we apply a uniform promotion $(s_p,s_q)$ to both groups, then
		\begin{equation}
			\label{eq:p_q_s_het}
			p^k(s_{p}):=p_0^k+b_{p}^k\sqrt{s_p}, \qquad q^k(s_{q}):=q_0^k+b_{q}^k\sqrt{s_q}, \qquad k=1,2, \qquad s_p,s_q\geq 0,
		\end{equation}
		where $b_{q}^k\sqrt{s_q}$ represents the extent to which promotional spending amplifies the sensitivity to peer influence of individuals of type~$k$.
		
		Let us denote the expected adoption level by
		$
		f^{\rm complete-het}(t)=f_{1}^{\rm complete}(t)+f_{2}^{\rm complete}(t),
		$
		where $f_{k}^{\rm complete}(t):=\frac{1}{M}\sum_{j\in\mathcal{M}_k}\E[X_j]$ is the expected population fraction of adopters from group $\mathcal{M}_k$. As in the homogeneous case, as $M\to\infty$, the heterogeneous complete network model approaches the corresponding heterogeneous compartmental model:
		\begin{theorem}[\cite{Bass-monotone-convergence-23}]
			\label{thm:groups_convergence}
			Consider the Bass model~{\rm (\ref{eqs:Bass-model},~\ref{eq:het_conditions})} on a heterogeneous complete network with two homogeneous groups of equal size and with time-dependent parameters. Then $\lim\limits_{M\to\infty}f^{\rm complete-het}=f_1^{\rm compart-het}+f_2^{\rm compart-het}$, where $\big\{f_k^{\rm compart-het}\big\}_{k=1}^2$ are the solutions of the equation
			\begin{equation}
				\label{eq:groups_limit}
				\frac{df_k}{dt}=\Big(\frac{1}{2}-f_k\Big)
				\Big(p^k(t)+q^k(t)(f_1+f_2)\Big),\qquad f_k(0)=0,\qquad k=1,2.
			\end{equation}
		\end{theorem}
		
		The exact reduction of the aggregate dynamics to two ODEs enables us to formulate the optimal promotion as a boundary-value problem with four equations (see Theorem~\ref{thm:opt_group_same}). Solving these equations numerically suggests that the heterogeneity of the population does not qualitatively affect the optimal promotion. For example, In Figure~\ref{fig:pq_opt_groups} we plot the optimal promotion when the first group has the same parameters as in Figure~\ref{fig:pq_opt}, and all the parameters of the second group are doubled. As always, $s_p^{\rm opt}(t)$ decreases, whereas $s_q^{\rm opt}(t)$ increases from zero to a global maximum and then decreases. The relative sizes of $s_p^{\rm opt}(t)$ and $s_q^{\rm opt}(t)$ are similar to those in the homogeneous case, i.e., $s_p^{\rm opt}(t)$ is initially dominant, but as more individuals adopt, $s_q^{\rm opt}(t)$ becomes dominant. Additionally, $s_p^{\rm opt}(t)+s_q^{\rm opt}(t)$ approaches zero when $f^{\rm opt}$ approaches one. Hence, in this case, the heterogeneity of the network does not qualitatively affect the optimal promotion.

		\begin{figure}[ht!]
			\centering
			\scalebox{0.65}{\includegraphics{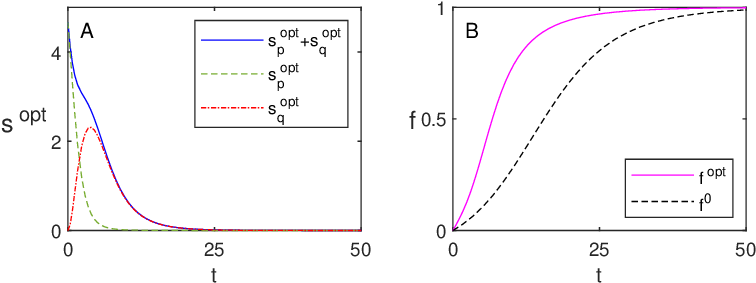}}
			\caption{Same as Figure~\ref{fig:pq_opt_M2}, on a heterogeneous network~\eqref{eq:het_conditions} with two homogeneous groups where $p_0^1=0.01$, $p_0^2=0.02$, $q_0^1=0.1$, $q_0^2=0.2$, $b_{p}^1=0.01$, $b_{p}^2=0.02$, $b_{q^1}=0.1$, and $b_{q^2}=0.2$. Here $\Delta\Pi^{\rm opt}_{\rm rel}\approx 6.2\%$.}
			\label{fig:pq_opt_groups}
		\end{figure}

		\subsection{Targeted promotions}
		
		Assume that each group can be targeted with a distinct promotion policy $(s_p^k, s_q^k)$ for $k=1,2$. Then
		\begin{equation}
			\label{eq:p_q_s_diff}
			p_k(s_{p}^k):=p_0^k+b_{p}^k\sqrt{s_{p}^k}, \qquad q_k(s_{q}^k):=q_0^k+b_{q}^k\sqrt{s_{q}^k}, \qquad k=1,2.
		\end{equation}
		The exact reduction of the aggregate dynamics to two ODEs enables us to formulate the optimal promotion as a boundary-value problem with four equations (see Theorem~\ref{thm:opt_group_diff}). Figure~\ref{fig:pq_opt_groups_diff} illustrates the targeted optimal promotion policies for the same heterogeneous population considered in Figure~\ref{fig:pq_opt_groups}. We observe that the promotion level for group~2 is significantly higher than under the uniform policy, which itself exceeds the promotion level for group~1, i.e., $s_{p}^{2, \rm opt}(t)+s_{q}^{2, \rm opt}(t)> s_p^{\rm opt}+s_q^{\rm opt}>s_{p}^{1, \rm opt}(t)+s_{q}^{1, \rm opt}(t)$, reflecting the larger values of $b_p^2$ and $b_q^2$. Hence, the faster adoption in group~2 compensates for the lower promotion effort allocated to group~1. Interestingly, {\em the effectiveness of the heterogeneous optimal promotion $(\Delta\Pi^{\rm opt}_{\rm rel}\approx 6.3\%)$ is only marginally higher than of a homogeneous one $(\Delta\Pi^{\rm opt}_{\rm rel}\approx 6.2\%)$}. 
		

		\begin{figure}[ht!]
			\centering
			\scalebox{0.65}{\includegraphics{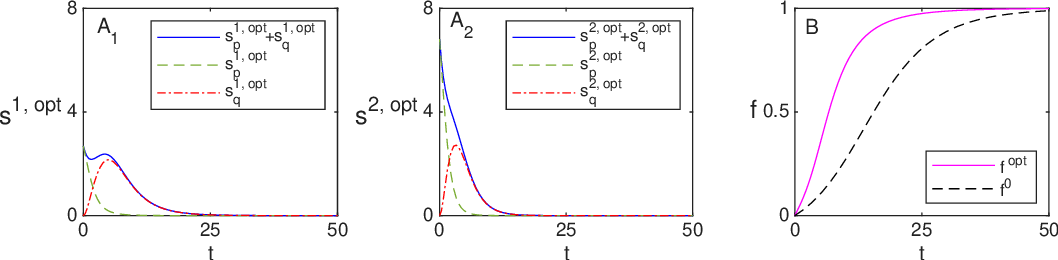}}
			\caption{{Targeted promotion for the heterogeneous population considered in Figure~\ref{fig:pq_opt_groups}.} Panels $A_1$ and $A_2$ show the optimal promotion policies for groups~1 and~2, respectively, while panel~$B$ displays the corresponding adoption dynamics. Here, $\Delta\Pi^{\rm opt}_{\rm rel}\approx 6.3\%$.}
			\label{fig:pq_opt_groups_diff}
		\end{figure}

			\subsection{One-sided spillover effects}
			
			A different two-group setup arises when promotion directed at group~1 also influences individuals in group~2, whereas promotion directed at group~2 affects only that group. This situation naturally occurs, for example, when group~1 consists of native speakers and group~2 represents a linguistic minority. In this case, promotion aimed at group~1 is conducted in the dominant (formal) language, which is also understood by the minority, while promotion aimed at group~2 is delivered in the minority language and is not understood by the native speakers. Therefore, 
			%
			\begin{equation}
				\label{eq:innovator}
				\begin{aligned}
					&p_1(s_p^1):=p_0^1+b_p^1\sqrt{s_p^1}, \quad &&q_1(s_q^1):=q_0^1+b_q^1\sqrt{s_q^1},\\		
					&p_2(s_p^2):=p_0^2+b_p^2\sqrt{s_p^2}+b_p^{1\to 2}\sqrt{s_p^1}, \quad &&q_2(s_q^2):=q_0^2+b_q^2\sqrt{s_q^2}+b_q^{1\to 2}\sqrt{s_q^1},
				\end{aligned}
			\end{equation}
			where the terms $b_p^{1\to 2}\sqrt{s_p^1}$ and $b_q^{1\to 2}\sqrt{s_q^1}$ represent the spillover effect of promotion directed at group~1 on individuals in group~2.
			%
			Figure~\ref{fig:inn} shows the optimal promotion for the same heterogeneous population considered in Figure~\ref{fig:pq_opt_groups_diff}, with the addition of the spillover effect. In this case, $s_{p}^{1,\rm opt}(t)+ s_{q}^{1,\rm opt}(t)> s_{p}^{2,\rm opt}(t)+s_{q}^{1,\rm opt}(t)$, which is the opposite of the behavior observed in Figure~\ref{fig:pq_opt_groups_diff}. Indeed, although group~1 benefits less directly from promotion, its promotion now affects both groups, creating a stronger incentive to allocate resources to it. Furthermore, the effectiveness of the targeted optimal promotion $(\Delta\Pi^{\rm opt}_{\rm rel}\approx 8.5\%)$ is much higher than in the case without spillover  effects $(\Delta\Pi^{\rm opt}_{\rm rel}\approx 6.3\%)$, since the any spending level $(s_p^1,s_q^1)$ generates a larger overall impact.
			
			\begin{figure}[ht!]
				\centering
				\scalebox{0.5}{\includegraphics{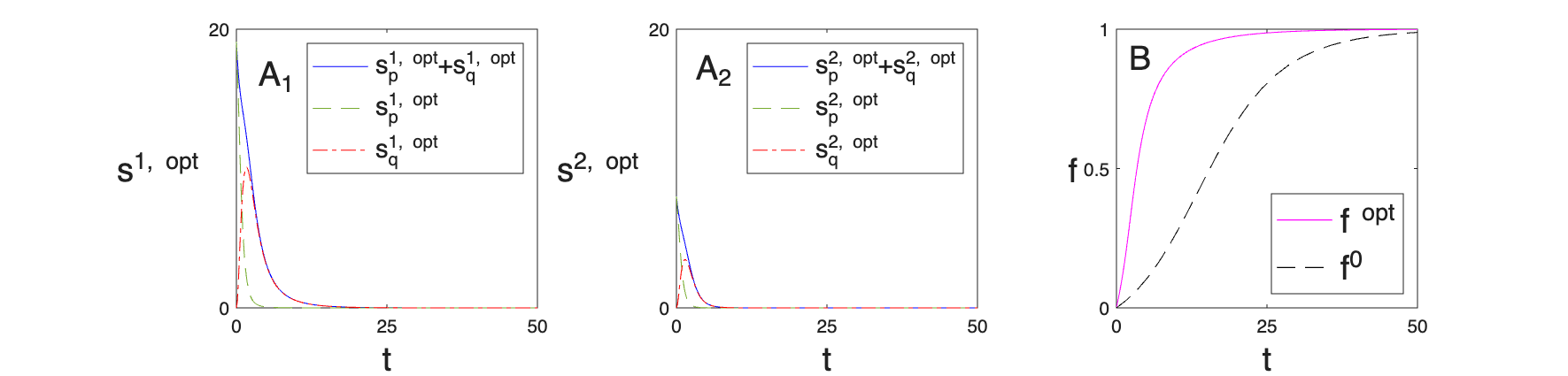}}
				\caption{Same as in Figure~\ref{fig:pq_opt_groups_diff}, but with one-sided spillover effects, see~\eqref{eq:innovator}. Here $b_p^{1\to 2}=b_{p}^2$,  $b_q^{1\to 2}=b_{q}^2$, and $\Delta\Pi^{\rm opt}_{\rm rel}\approx 8.5\%$.}
				\label{fig:inn}
			\end{figure}
			
			\section{Suboptimality of non-uniform promotions}
			\label{sec:suboptimal}
			
			In a homogeneous infinite complete network, is it optimal to promote uniformly across the entire population? The answer is not immediately obvious, since one could instead concentrate advertising efforts on one half of the population and rely on internal influence to drive adoption in the other half. To address this question, we artificially partition the population into two identical subgroups, each of size~$1/2$, which may be subject to different promotion policies  $(s_p^k, s_q^k)$ for $k=1,2$. Let $f_k(t) \in [0,1/2]$ denote the adoption level in subgroup $k=1,2$ at time $t$. By Theorem~\ref{thm:groups_convergence}, the adoption dynamics are given by
			\begin{equation}
				\label{eq:f_k-proof}
				\frac{df_k}{dt}(t) = \Big(\frac{1}{2} - f_k\Big) \Big( p_0 + b_p\sqrt{s_p^k(t)} + \bigl(q_0 +b_q \sqrt{s_q^k(t)}\bigr)\bigl(f_1+f_2\bigr) \Big),
				\quad f_k(0)=0, \qquad k=1,2.
			\end{equation}
			The firm discounted profit is
			\[
			\Pi[s_p^1,s_p^2,s_q^1,s_q^2]
			:= \int_0^T e^{-\theta t}
			\left(
			\gamma \left(\frac{df_1}{dt}(t)+\frac{df_2}{dt}(t)\right)
			- \left(\frac{s_p^1(t)+s_p^2(t)}{2} + \frac{s_q^1(t)+s_q^2(t)}{2}\right)
			\right)dt.
			\]
			We compare this policy with a uniform promotion defined as
			\[
			\sqrt{s_p(t)} := \frac{\sqrt{s_p^1(t)} + \sqrt{s_p^2(t)}}{2},
			\qquad
			\sqrt{s_q(t)} := \frac{\sqrt{s_q^1(t)} + \sqrt{s_q^2(t)}}{2}.
			\]
			Let $f(t)$ be the corresponding adoption trajectory. Then 
			\[
			\frac{df}{dt}(t)
			= (1-f)
			\left(p_0 + b_p\sqrt{s_p(t)} + \bigl(q_0+b_q\sqrt{s_q(t)}\bigr)f\right),
			\qquad f(0)=0.
			\]
			The associated firm profit is
			\[
			\Pi[s_p,s_q]
			:= \int_0^T e^{-\theta t}
			\left(
			\gamma \frac{df}{dt}(t) - \bigl(s_p(t)+s_q(t)\bigr)
			\right)dt.
			\]

			\begin{theorem}
				\label{thm:non_uniform}
				Assume $\gamma > 0$, $\theta \ge 0$, and $0<T\le \infty$. For any nonnegative piecewise-continuous promotions $\{s_p^1(t), s_p^2(t), s_q^1(t), s_q^2(t)\}$, the uniform promotion dominates the non-uniform policy, that is, 
				$
				\Pi[s_p^1,s_p^2,s_q^1,s_q^2] \le \Pi[s_p,s_q].
				$ 
				Furthermore, an equality holds if and only if $s_p^1(t)=s_p^2(t)$ and $s_q^1(t) = s_q^2(t)$ for all $t \in [0, T]$. 
			\end{theorem}
			\begin{proof}
				See Appendix~\ref{app:thm:non_uniform}.
			\end{proof}
			
			Since Theorem~\ref{thm:non_uniform} applies to all non-uniform promotion policies, it follows in particular that the optimal promotion strategy is uniform:
			\begin{corollary}
				\label{cor:non_uniform}
				In a homogeneous infinite complete network, it is suboptimal to implement a different promotion to each half of the population.
			\end{corollary}
			
			\section{Alternative models for the impact of promotions}
			\label{sec:Alternative-models}
			
			Thus far, we have assumed that~$p$ and~$q$ depend on promotional spending through the square-root functional form in~\eqref{eq:p_q_s}. In this section, we examine the robustness of this assumption by considering several alternative specifications. A natural generalization is given by a 
			$k$th root representation:
			\begin{equation}
				\label{eq:p_q_s_1/n}
				p(s_p) := p_0 + b_p \sqrt[k]{s_p}, \quad
				q(s_q) := q_0 + b_q \sqrt[k]{s_q}, \qquad
				s_p, s_q \geq 0, \quad k = 2,3,\dots
			\end{equation}
			Alternatively, following \cite{Horsky-83}, one may consider a logarithmic form:
			\begin{equation}
				\label{eq:p_q_s_ln}
				p(s_p) := p_0 + b_p \ln(1 + s_p), \qquad
				q(s_q) := q_0 + b_q \ln(1 + s_q).
			\end{equation}
			Both forms are increasing and concave in promotional spending, thereby capturing diminishing marginal returns.
			
			Applying the methodology developed in this manuscript to these alternative spend-to-influence functional forms is straightforward. In particular, in the case of {\em infinite complete networks}, Corollary~\ref{cor:opt_complete_pq} extends as follows.
			
			\begin{corollary}
				\label{cor:opt_complete_1/n}
				Assume the conditions of Theorem~{\rm \ref{lem:opt_complete_pq}}. 
				\begin{enumerate}
					\item If the effect of promotion is given by~\eqref{eq:p_q_s_1/n}, then equation~\eqref{eq:H_inf} simplifies to
					\begin{equation}
						\label{eq:s_lambda_1/n}
						s^{\rm opt}_p(t)
						=
						\left(\frac{b_p(1-f)(\Psi e^{\theta t}+\gamma)}{k}\right)^{\frac{k}{k-1}},
						\quad
						s^{\rm opt}_q(t)
						=
						\left(\frac{b_q}{b_p}f\right)^{\frac{k}{k-1}} s^{\rm opt}_p(t),
						\qquad k=2,3,\dots
					\end{equation}
					
					\item If the effect of promotion is given by~\eqref{eq:p_q_s_ln}, then equation~\eqref{eq:H_inf} simplifies to
					\begin{equation}
						\label{eq:s_lambda_ln}
						s^{\rm opt}_p(t)
						=
						\Big( b_p(1-f)(\Psi e^{\theta t}+\gamma) - 1 \Big)^+,
						\quad
						s^{\rm opt}_q(t)
						=
						\Big( b_q(1-f)f(\Psi e^{\theta t}+\gamma) - 1 \Big)^+,
					\end{equation}
				\end{enumerate}
				where $x^+ := \max\{x,0\}$ denotes the positive part of $x$.
			\end{corollary}
			\begin{proof}
				See Appendix~\ref{app:cor:opt_complete_1/n}.
			\end{proof}

			\begin{figure}[ht!]
				\centering
				\scalebox{0.35}{\includegraphics{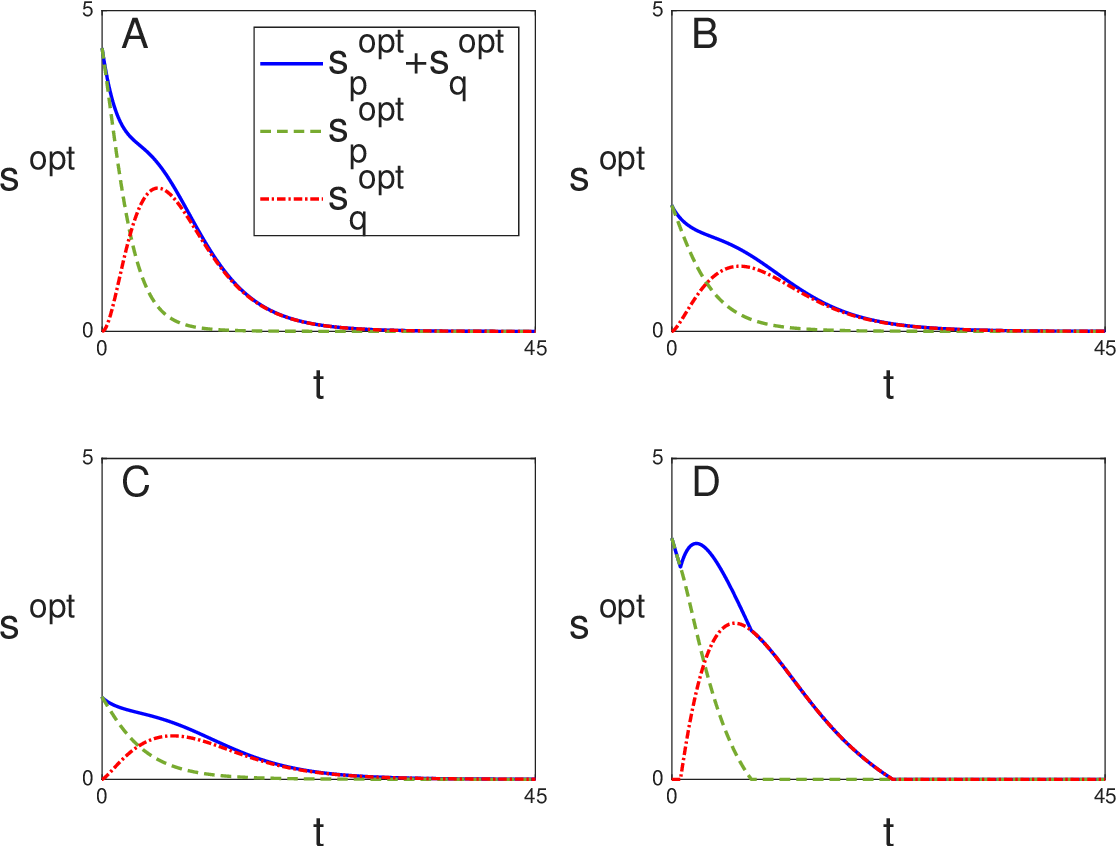}}
				\caption{Optimal promotion on an infinite complete network. 
					(A) Same as Figure~\ref{fig:pq_opt}. 
					(B) Same as (A), but with $p(s_p)$ and $q(s_q)$ given by~\eqref{eq:p_q_s_1/n} with $k=3$ ($\Delta\Pi^{\rm opt}_{\rm rel}=8.5\%$). 
					(C) Same as (A), but with $k=4$ ($\Delta\Pi^{\rm opt}_{\rm rel}=8.8\%$). 
					(D) Same as (A), but with $p(s_p)$ and $q(s_q)$ given by the logarithmic relation~\eqref{eq:p_q_s_ln} ($\Delta\Pi^{\rm opt}_{\rm rel}=6.0\%$).}
				\label{fig:robust}
			\end{figure}
			
			Plotting the optimal promotion strategies for these alternative models in Figure~\ref{fig:robust} reveals the same qualitative behavior observed in Figure~\ref{fig:pq_opt}: $s_p^{\rm opt}(t)$ decreases monotonically over time, whereas $s_q^{\rm opt}(t)$ initially increases from zero to a peak and subsequently decreases.
			
			Similarly, in the case of {\em finite complete networks}, Corollary~\ref{cor:opt_complete_M} extends as follows.
			
			\begin{corollary}
				\label{cor:opt_complete_M_1/n}
				Assume the conditions of Theorem~{\rm \ref{lem:opt_M}}. 
				\begin{enumerate}
					\item If the effect of promotion is given by~\eqref{eq:p_q_s_1/n}, then equation~\eqref{eq:H_comp} simplifies to 
					\begin{equation}
						\label{eq:s_lambda_pq_M_1/n}
						\begin{aligned}
							s^{\rm opt}_p(t)
							&=
							\left(\frac{b_p}{k}\Big(\gamma[S]-e^{\theta t}\sum_{n=1}^{M}\Psi_n [S^n]\Big)\right)^{\frac{k}{k-1}},\\
							s^{\rm opt}_q(t)
							&=
							\left(\frac{b_q}{k}\Big(\gamma\left([S]-[S^2]\right)-e^{\theta t}\sum_{n=1}^{M-1}\Psi_n n \frac{M-n}{M-1}\big([S^n]-[S^{n+1}]\big)\Big)\right)^{\frac{k}{k-1}}.
						\end{aligned}
					\end{equation}
					
					\item If the effect of promotion is given by~\eqref{eq:p_q_s_ln}, then equation~\eqref{eq:H_comp} simplifies to
					\begin{equation}
						\label{eq:s_lambda_pq_M_ln}
						\begin{aligned}
							s^{\rm opt}_p(t)
							&=
							\Big(b_p\Big(\gamma[S]-e^{\theta t}\sum_{n=1}^{M}\Psi_n [S^n]\Big)-1 \Big)^+,\\
							s^{\rm opt}_q(t)
							&=
							\Big(b_q\Big(\gamma\left([S]-[S^2]\right)-e^{\theta t}\sum_{n=1}^{M-1}\Psi_n n \frac{M-n}{M-1}\big([S^n]-[S^{n+1}]\big)\Big)-1\Big)^+.
						\end{aligned}
					\end{equation}
					
				\end{enumerate}
			\end{corollary}
			%
		%
	%
\begin{proof}
	See Appendix~\ref{app:cor:opt_complete_M_1/n}.
\end{proof}

\begin{figure}[ht!]
	\centering
	\scalebox{0.4}{\includegraphics{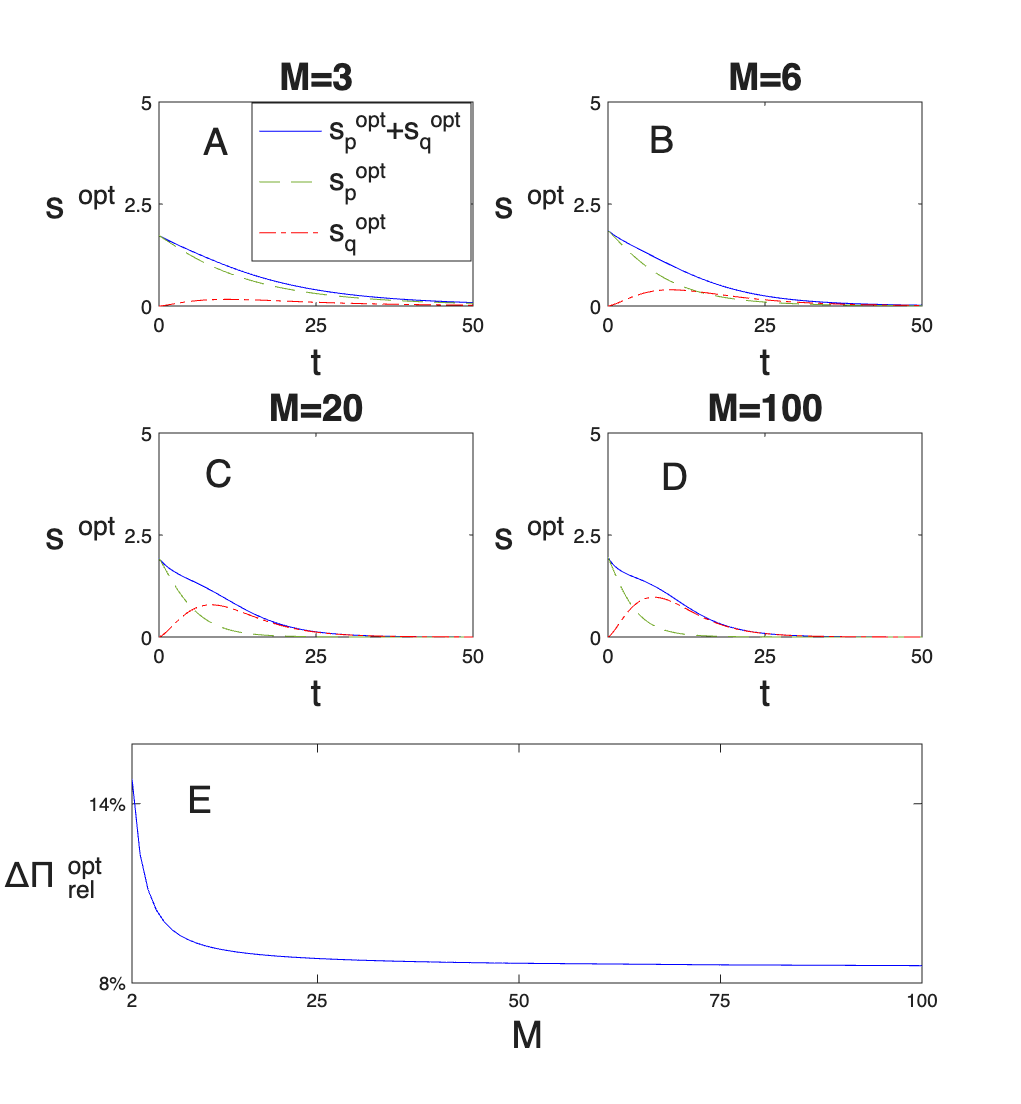}}
	\caption{Same as Figure~\ref{fig:pq_opt_M}, but with~$p(s_p)$ and~$q(s_q)$ given by~\eqref{eq:s_lambda_1/n} with $k=3$.}
	\label{fig:pq_opt_M_1/n}
\end{figure}

\begin{figure}[ht!]
	\centering
	\scalebox{0.4}{\includegraphics{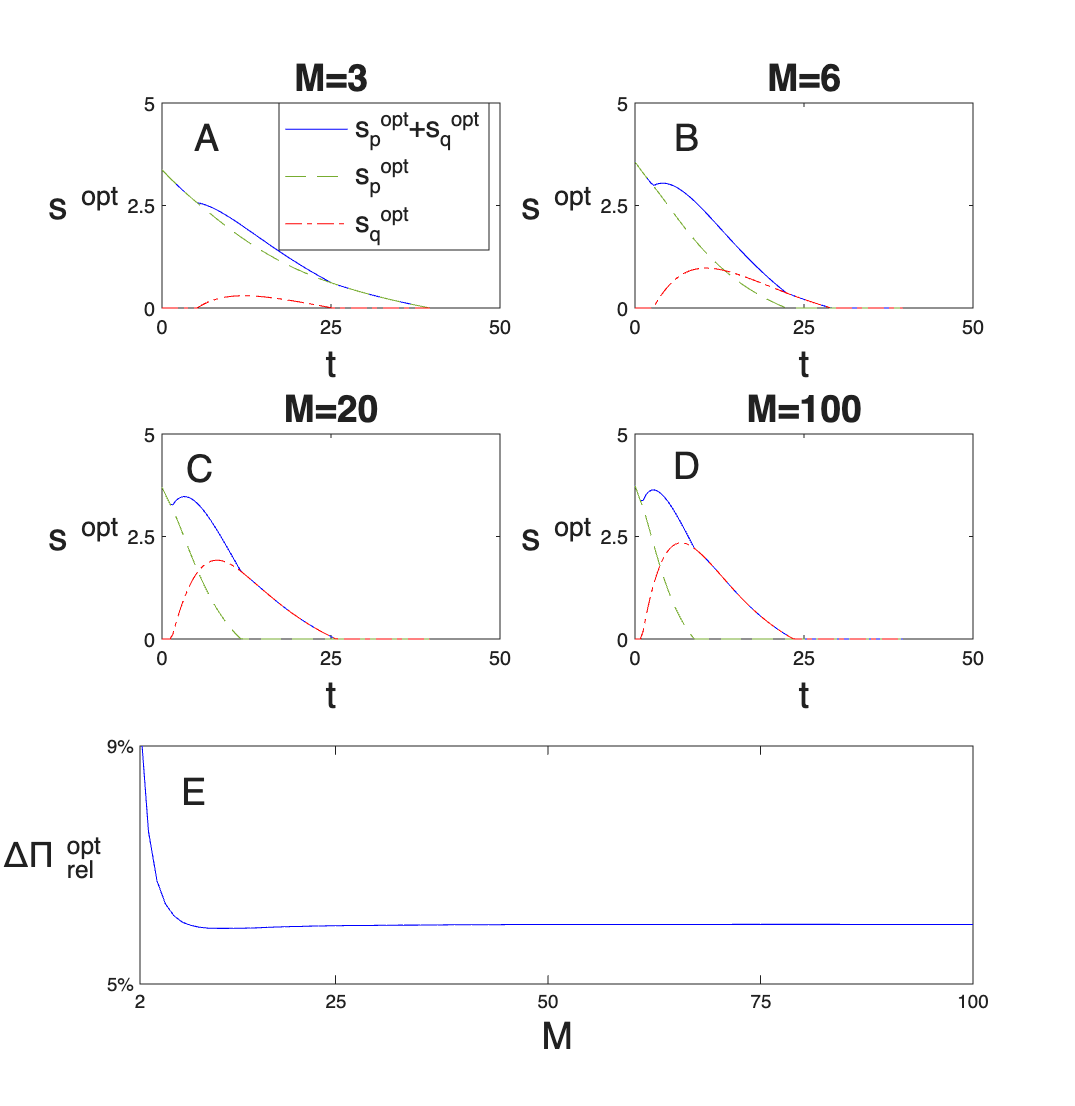}}
	\caption{Same as Figure~\ref{fig:pq_opt_M}, but with~$p(s_p)$ and~$q(s_q)$ given by~\eqref{eq:s_lambda_ln}.}
	\label{fig:pq_opt_M_ln}
\end{figure}

Plotting the optimal promotion strategies for these alternative models in Figures~\ref{fig:pq_opt_M_1/n} and~\ref{fig:pq_opt_M_ln} reveals the same qualitative behavior observed in Figure~\ref{fig:pq_opt_M}: $s_p^{\rm opt}(t)$ decreases monotonically over time, while $s_q^{\rm opt}(t)$ initially increases from zero to a peak and subsequently decreases. Moreover, as $M$ increases, the relative magnitude of $s_q^{\rm opt}(t)$ becomes larger, and the relative impact of the optimal promotion, $\Delta\Pi^{\rm opt}_{\rm rel}$, decreases.

There is, however, a {\em novel qualitative behavior} observed in Figures~\ref{fig:robust}D and~\ref{fig:pq_opt_M_ln}: when $p(s_p)$ and $q(s_q)$ follow the logarithmic relation~\eqref{eq:s_lambda_ln}, there exist time intervals during which either $s_p^{\rm opt}(t)$ or $s_q^{\rm opt}(t)$ is identically zero. To understand why this behavior arises under the logarithmic relation~\eqref{eq:s_lambda_ln}, but not under the $k$th root relation~\eqref{eq:p_q_s_1/n}, consider the profit rate for $0 < s_p \ll 1$ (the analysis for $s_q$ is analogous). For example, in the case of infinite complete networks, it follows from~\eqref{eq:profit-b} and~\eqref{eq:bvp_pq} that 
\[
\pi(s_p) - \pi(0) \sim e^{-\theta t}\Big(\gamma(1-f)\big(p(s_p)-p_0\big)-s_p \Big).
\]
In the case of the $k$-root relation,
\[
\gamma(1-f)\big(p(s_p)-p_0\big)-s_p 
= \gamma(1-f)b_p \sqrt[k]{s_p} - s_p > 0, 
\qquad 0 < s_p \ll 1.
\]
so even an arbitrarily small positive promotion level strictly improves profit relative to zero, implying that $s_p=0$ is never optimal. In contrast, under the logarithmic relation,
\[
\gamma(1-f)\big(p(s_p)-p_0\big)-s_p 
\sim \Big(\gamma(1-f)b_p - 1\Big)s_p,  
\qquad 0 < s_p \ll 1.
\]
This expression becomes negative when $\gamma(1-f)b_p < 1$. Therefore, as $f \to 1$, it is optimal to set the promotion level to zero.

\section{Feedback (closed-loop) optimal promotions}
\label{sec:closed-loop}

So far, we have considered promotional strategies in which the optimal control ${\bf s}^{\rm opt}(t)$ is predetermined and does not depend on the realized adoption process. In practice, however, it is desirable to allow the promotional strategy to respond to the current state of the system. For example, for the compartmental Bass model~\eqref{eq:Bass_intro}, the optimal feedback (closed-loop) control $s_p^{\text{opt-feedback}}$ depends both on time and on the current adoption level. It satisfies the optimality condition
\begin{equation*}
	s_p^{\text{opt-feedback}}(t,f)
	:=
	\arg\max_{s_p(t,f)}
	\int_{t_0}^{T}
	\Big(
	\gamma \frac{df}{dt}
	-
	s_p(t,f(t))
	\Big)
	e^{-\theta t}
	\, dt,
\end{equation*}
where
\begin{equation*}
	\frac{df}{dt}
	=
	(1 - f)\Big( p + b_p \sqrt{s_p(t,f(t))} + qf \Big),
	\qquad
	f(t_0) = f_0,
\end{equation*}
for all $0 \le t_0 < T$ and $0 \le f_0 < 1$. The computation of optimal feedback strategies is typically carried out using the Hamilton–Jacobi–Bellman (HJB) equation (see, e.g.,~\cite{kamienschwartz}). To the best of our knowledge, feedback-optimal promotion policies for compartmental Bass models have only been studied by Fruchter et al.~\cite{Fruchter-22}, who employed the HJB framework. 

We now turn to feedback-optimal promotions for the Bass model on networks. In this setting, the state of the system at time $t$ is given by the realized adoption level,
\(
\phi(t) := \frac{1}{M} \sum_{j=1}^{M} X_j(t),
\)
where the random variables $\{X_j(t)\}$ evolve according to~\eqref{eqs:Bass-model}. The optimal feedback strategy ${\bf s}^{\text{opt-feedback}}$ therefore depends on both time~$t$ and current adoption level~$\phi(t)$, and is defined as
\begin{equation*}
	{\bf s}^{\text{opt-feedback}}(t,\phi)
	:=
	\arg\max_{{\bf s}(t,\phi)}
	\int_{t_0}^{T} 
	e^{-\theta t} \Big(
	\gamma \frac{df}{dt}
	-
	\mathbb{E}\big[
	s_p(t,\phi(t)) + s_q(t,\phi(t))
	\big]
	\Big)
	\, dt,
\end{equation*}
where $f(t) := \mathbb{E}[\phi(t)]$. Thus, at any given time~$t_0$ and adoption level $\phi_0$, the optimal promotion policy is the one that maximizes the expected discounted profit over the remaining time horizon.

To apply the HJB methodology, one requires an ODE for~$\frac{df}{dt}$ on networks, analogous to equation~\eqref{eq:bvp_pq} in the open-loop setting. Deriving such an equation, however, is not a straightforward extension of our open-loop approach. Indeed, under feedback control, the adoption rate of node~$j$ changes from~\eqref{eq:pjqj} to
\begin{equation}
	\label{eq:lambda-feedback}
	\lambda_j^{\text{feedback}}(t)
	:=
	p_j\big(s_p(t,\phi(t))\big)
	+
	\sum_{k=1}^{M}
	q_{k\to j}\big(s_q(t,\phi(t))\big) X_k(t).
\end{equation}
Hence, the adoption rate depends nonlinearly on the $\{X_j(t)\}$ through the state variable~$\phi(t)$. Consequently, the system of master equations~\eqref{eq:master_general}, whose derivation relied on the linear dependencen~$\{X_j\}$, is no longer valid. One must therefore first derive a new system of master equations for the network Bass model~\eqref{eqs:Bass-model} with the feedback-dependent adoption rate~\eqref{eq:lambda-feedback}. After obtaining these equations, it would be necessary to derive reduced master equations for structured networks (e.g., complete graphs or line networks) in order to compute
$
\frac{df}{dt}:= \frac{d}{dt}\mathbb{E}[\phi(t)].
$
An additional complication arises from the fact that
\[
\mathbb{E}[{\bf s}(t,\phi(t))]
\neq
{\bf s}(t,\mathbb{E}[\phi(t)])
=
{\bf s}(t,f(t)).
\]
In summary, computing feedback-optimal strategies for the Bass model on networks constitutes an important and challenging problem. Addressing it, however, requires substantial additional analysis that lies beyond the scope of the present manuscript.

	\section{An alternative profit formulation}
	\label{sec:alternative}
	
	In the profit expression~\eqref{eq:profit-b} considered thus far, the costs of promotions in~$p$ and~$q$ are independent of the numbers of adopters and non-adopters. Implicitly, this assumes that all consumers are targeted by these promotions, regardless of their adoption status. In some cases, however, advertising costs may depend on the composition of the consumer base. For example, targeted advertising aimed specifically at non-adopters generates costs proportional to the number of non-adopters. Likewise, the cost of a promotion in~$q$ may scale with the number of adopter–non-adopter interactions, as in referral programs that offer discounts for successful recommendations. In this section, we show that our methodology can be extended to accommodate such alternative profit specifications. For simplicity, we restrict attention to infinite complete networks.

		Let $s_p$ denote the spending level on external promotions {\em per non-adopter} (rather than per individual). The total spending rate on external promotions is therefore proportional to the fraction of nonadopters and is given by $(1-f)s_p$. Similarly, Let~$s_q$ denote the spending level on internal promotions {\em per adopter-nonadopter interaction}. The corresponding spending rate is then proportional to the product of the fractions of adopters and nonadopters and given by  $f(1-f)s_q$. Under this specification, the instantaneous profit rate becomes
		\begin{equation}
			\label{eq:scale2}
			\pi({\bf s}(t)) = e^{-\theta t}\Big(\gamma \frac{df}{dt}-\big((1-f)s_p(t)+f(1-f)s_q(t)\big) \Big).
		\end{equation}
		Importantly, these modifications affect only the cost structure and do not alter the underlying Bass adoption dynamics.
		
		\begin{theorem}
			\label{thm:opt_complete_scale2}
			Consider the Bass model~\eqref{eqs:Bass-model} on the homogeneous complete network~\eqref{eq:complete} in the limit $M \to \infty$. Suppose that the promotional effects are given by~\eqref{eq:p_q_s} and that the profit rate is defined by~\eqref{eq:scale2}. Then the optimal promotion policy ${\bf s}^{\rm opt} := (s_p^{\rm opt}, s_q^{\rm opt})$, see~\eqref{eq:profit-a}, is given by 
			\begin{equation}
				\label{eq:s_lambda_pq_scale2}
				s^{\rm opt}_p(t)=\frac{b_p^2}{4}\Big((\Psi(t) e^{\theta t}+\gamma)\Big)^2,\qquad
				s^{\rm opt}_q(t)=\frac{b_q^2}{b_p^2}s^{\rm opt}_p(t),
			\end{equation}
			%
			%
			where $f(t)$ and $\Psi(t)$ solve the boundary-value problem
			\begin{equation*}
				\begin{aligned}
					\frac{df}{dt} &= (1-f)\Big(p(s_p(t)) + q(s_q(t))f\Big),
					\qquad && f(0) = 0, \\
					\frac{d\Psi}{dt} &= \Big(\gamma e^{-\theta t} + \Psi\Big)\Big(p(s_p(t)) + q(s_q(t))(2f - 1)\Big) 
					- \Big(s_p(t) + (2f - 1)s_q(t)\Big)e^{-\theta t},
					\qquad && \Psi(T) = 0.
				\end{aligned}
			\end{equation*}
		\end{theorem}
		\begin{proof}
			See Appendix~\ref{app:opt_complete_scale2}.
		\end{proof}	  
		
		\begin{figure}[ht!]
			\centering
			\scalebox{0.5}{\includegraphics{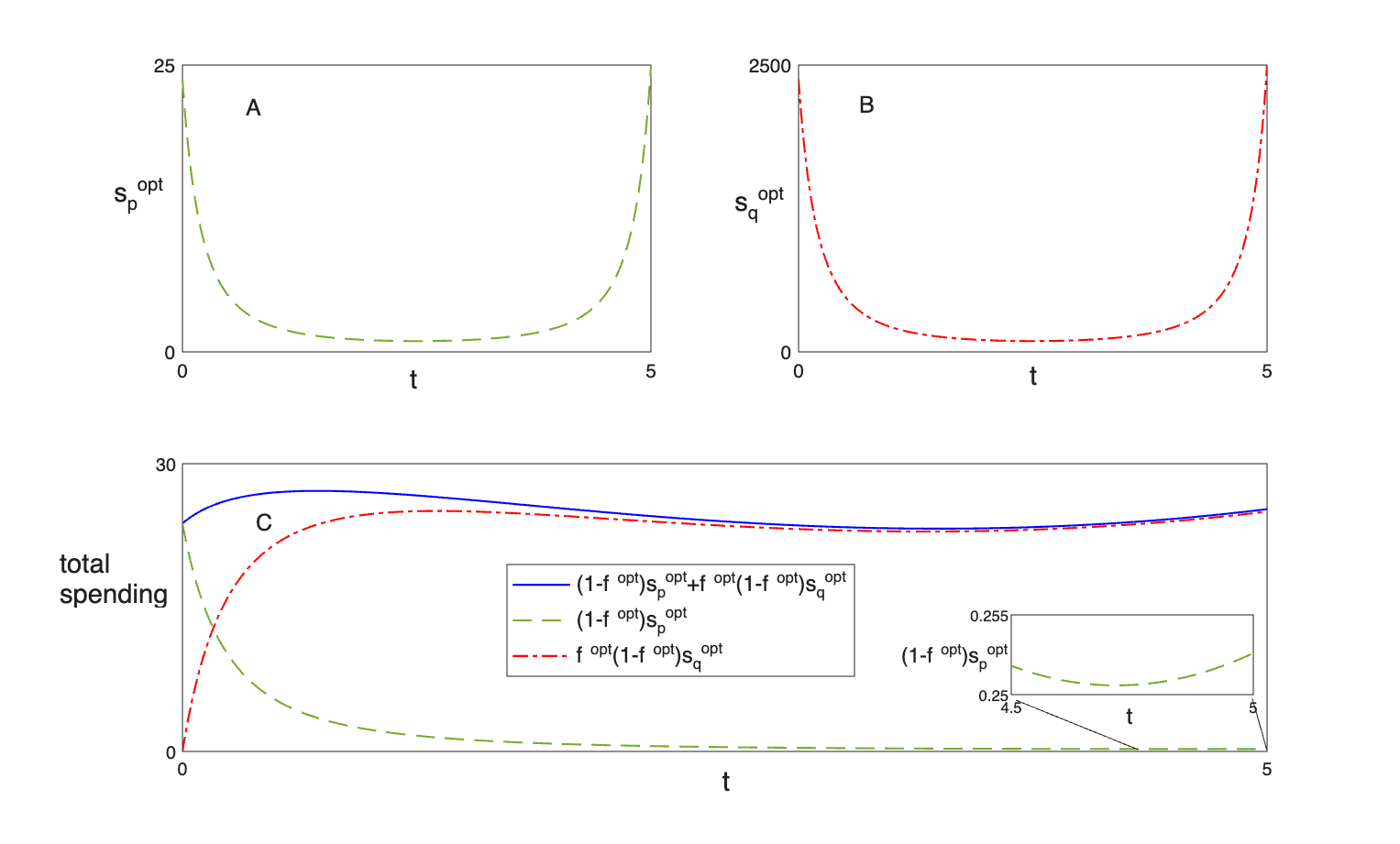}}
			\caption{(A) The optimal external promotion level as a function of time for the same Bass model as in Figure~\ref{fig:pq_opt}, but with the profit rate given by~\eqref{eq:scale2}. (B) Same as in (A), for the internal promotion level. (C) The corresponding spending rates on external promotion (dashed), internal promotion (dash-dot), and their total (solid).  Inset shows that $(1-f^{\rm opt})s_p^{\rm opt}$ undergoes a late-term surge.}
			\label{fig:prop2}
			\label{fig:prop2}
		\end{figure}
		
		
		The dynamics of $s_p^{\rm opt}(t)$ and $s_q^{\rm opt}(t)$ are identical up to a multiplicative constant; see~\eqref{eq:s_lambda_pq_scale2} and Figure~\ref{fig:prop2}. Both decline monotonically over most of the horizon. Intuitively, earlier adoptions are more valuable, as they have more time to generate subsequent adoptions through peer effects. As the end of the promotion period approaches, however, both $s_p^{\rm opt}(t)$ and $s_q^{\rm opt}(t)$ increase sharply. This late-stage surge reflects the diminishing value of future diffusion: as $t\to T$, the planner can no longer influence post-horizon adoption. Consequently, the optimal promotion becomes myopic (see Corollary~\ref{cor:end}), and is directed toward attracting consumers who would otherwise not adopt before the terminal time, thereby increasing the {\em expansion}.
		
		The total spending rate on external promotions, $(1-f(t))s_p(t)$, declines over most of the horizon as the pool of susceptible consumers shrinks. As $t\to T$, however, it rises due to the late-stage myopic surge discussed above.
		The total spending rate on internal promotions, $f(t)(1-f(t))s_q(t)$, is zero at $t=0$, since no adopters are initially present. It increases as the adopter base expands, reflecting the growing scope for peer influence, then declines as the susceptible population contracts. As $t\to T$, however, internal promotion spending rises again, driven by the same myopic effect. The key distinction relative to the original profit formulation (see Section~\ref{sec:init}) is precisely this late-stage surge. It arises because promotions target only non-adopters, so even last-minute conversions yield positive returns.

			\appendix
			\section{Proof of Theorem~\ref{thm:gen}}
			\label{app:thm:gen}
			The expected adoption level is 
			\begin{equation}
				\label{eq:exp_adoption}
				f(t)=1-\frac{1}{M}\sum_{j=1}^{M}[S_{j}](t),
			\end{equation}
			see~\eqref{eq:frac_adopters}.
			Therefore, \eqref{eq:profit-b} can be rewritten as
			\begin{equation}
				\label{eq:profit_general}
				\pi({\bf s}(t))=
				e^{-\theta t} \Big(-\frac{\gamma}{M} \sum_{j=1}^{M}\frac{d[S_j]}{dt}-\big(s_p(t)+s_q(t)\big) \Big).
			\end{equation}
			Let us recall Pontryagin's Maximum Principle from optimal control theory (see e.g.,~\cite{kamienschwartz}):
			
			\begin{theorem}
				\label{thm:pontryagin}
				Consider the maximization problem
				\begin{equation}
					\label{eq:prob}
					{\bf s}^{\rm opt}(t)={\rm argmax}_{{\bf s}(t)}\Pi\left[{\bf s}(t)\right], \qquad \Pi:=\int_{t=0}^{T}\pi(t, {\bf x}(t), {\bf s}(t))\, dt,
				\end{equation}
				where ${\bf x}(t):=\left(x_1(t),\dots,x_n(t)\right): [0, T]\to\mathbb{R}^n$ is the continuous and differentiable solution of
				\begin{subequations}
					\label{eqs:max_prob}
					\begin{equation}
						\label{eq:f_ode}
						\frac{d{\bf x}}{dt}={\bf g}(t, {\bf x}(t),{\bf s}(t)), \qquad {\bf x}(0)={\bf x}_0,
					\end{equation}
					and ${\bf s}(t):=\left(s_1(t),\dots,s_m(t)\right)\in {\mathcal S}^m$ is the control function.
					Let
					$
					H(t,{\bf x},{\bf s},\boldsymbol{\Psi}):=\pi(t,{\bf x},{\bf s})+\sum_{k=1}^{n}\Psi_k(t)g_k(t,{\bf x},{\bf s}),
					$
					where $\boldsymbol{\Psi}(t)=\left(\Psi_1(t),\dots,\Psi_n(t)\right): [0, T]\to\mathbb{R}^n$, and $\Psi_i(t)$ satisfies
					\begin{equation}
						\label{eq:lambda_ode}
						\frac{d\Psi_i}{dt}=-\frac{\partial H}{\partial x_i}, \qquad \Psi_i(T)=0,\qquad i=1,\dots n.
					\end{equation}
				\end{subequations}
				In order for ${\bf s}^*(t)$ to solve~\eqref{eq:prob}, it is necessary that $\{{\bf x}^*(t),{\bf s}^*(t),\boldsymbol{\Psi}^*(t)\}$ be a solution of the boundary-value problem~\eqref{eqs:max_prob} that also satisfies the maximum condition $H\big(t;{\bf x}^*(t),{\bf s}^*(t),\boldsymbol{\Psi}^*(t)\big) = \max_{{\bf s}(t)\in {\mathcal S}^m}H\big(t;{\bf x}^*(t),{\bf s}(t),\boldsymbol{\Psi}^*(t)\big)$. In particular, if ${\bf s}^*(t)$ is an interior point of ${\mathcal S}^m$, then
				\begin{equation}
					\label{eq:s_condition}
					\frac{\partial H}{\partial s_j}\big(t,{\bf x}^*(t),{\bf s}^*(t),\boldsymbol{\Psi}^*(t)\big)=0, \qquad j=1,\dots,m.
				\end{equation}
			\end{theorem}
			
			Let us also recall the master equations for the Bass model:
			\begin{theorem}[\cite{fibich2021diffusion}]
				Consider the Bass model~\eqref{eqs:Bass-model}, where $\lambda_j(t)$ is given by~\eqref{eq:pjqj}. Let $\emptyset\neq \Omega \subset \mathcal{M}$, $\Omega^c:=M\setminus\Omega$, and let $p_\Omega$, $q_{k\to\Omega}$, and $[S_{\Omega,k}]$ be given by~\eqref{eq:notations}. 
				Then the master equation for $[S_{\Omega}]$ is
				\begin{equation}
					\label{eq:master_general}
					\frac{d[S_{\Omega}]}{dt}=-\Big(p_\Omega(s_p(t)) +\sum_{k\in\Omega^c}q_{k\to\Omega}(s_q(t))\Big)[S_{\Omega}]+\sum_{k\in\Omega^c}q_{k\to\Omega}(s_q(t))[S_{\Omega,k}],
					\quad [S_{\Omega}](0)=1.
				\end{equation}
			\end{theorem}
			
			\begin{proof}[Proof of Theorem~{\rm \ref{thm:gen}}]
				Let us apply Theorem~\ref{thm:pontryagin} with $\pi({\bf s}(t))$ given by~\eqref{eq:profit_general}, ${\bf x}:=\{[S_{\Omega}]\}_{\emptyset\neq\Omega\subset\mathcal{M}}$, and ${\bf g}:=\{g_{\Omega}\}_{\emptyset\neq\Omega\subset\mathcal{M}}$, where $g_\Omega$ right-hand side of~\eqref{eq:master_general}, and ${\bf s}(t):=(s_p(t),s_q(t))$. Then, using~\eqref{eq:exp_adoption},
				\begin{equation*}
					H=\Big(-\frac{\gamma}{M} \sum_{j=1}^{M}\frac{d[S_j]}{dt}-\left(s_p+s_q\Big) \right)e^{-\theta t}+\sum_{\emptyset\neq \Omega \subset \mathcal{M}}\Psi_\Omega\frac{d[S_\Omega]}{dt}.
				\end{equation*}
				Substituting the expressions for $\frac{d[S_j]}{dt}$ and $\frac{d[S_\Omega]}{dt}$ from~\eqref{eq:master_general}, we have
				\begin{align*}
					H&=\bigg(-\frac{\gamma}{M} \sum_{j=1}^{M}\Big(-\Big(p_j(s_p) +\sum_{k\neq j}q_{k\to j}(s_q)\Big)[S_j]+\sum_{k\neq j}q_{k\to j}(s_q)[S_{j,k}]\Big)-(s_p+s_q) \bigg)e^{-\theta t}\\
					&\qquad\qquad+\sum_{\emptyset\neq \Omega \subset \mathcal{M}}\Psi_\Omega\bigg(-\Big(p_\Omega(s_p) +\sum_{k\in\Omega^c}q_{k\to\Omega}(s_q)\Big)[S_{\Omega}]+\sum_{k\in\Omega^c}q_{k\to\Omega}(s_q)[S_{\Omega,k}]\bigg).
				\end{align*}
				The differential equation~\eqref{eq:lambda_prime_general} for $\Psi_\Omega$ follows from
				$
				\frac{d\Psi_\Omega}{dt}=-\frac{\partial H}{\partial [S_\Omega]}.
				$
			\end{proof}
			
			\section{Proof of Theorem~\ref{lem:opt_M}}
			\label{app:lem:opt_M}
			Subsituting~\eqref{eq:comp_reduction} in~\eqref{eq:s_prime_general}~and~\eqref{eq:c} gives~\eqref{eq:S_M}~and~\eqref{eq:H_comp}. The differential equation~\eqref{eq:lambda_M} for~$\Psi_n$ follows from
			$
			\frac{d\Psi_n}{dt}=-\frac{\partial H}{\partial [S^n]},
			$
			see~\eqref{eq:lambda_ode}.

			\section{Proof of Corollary~\ref{cor:opt_complete_M}}
			\label{app:cor:opt_complete_M}
			Substituting~\eqref{eqs:pq_opt_M} and~\eqref{eq:p_q_s} in~\eqref{eq:H_comp}, the equations $\frac{\partial H}{\partial s_p}=0$ and $\frac{\partial H}{\partial s_q}=0$ yield
			\begin{equation*}
				0=\frac{\gamma e^{-\theta t} b_p [S]}{2\sqrt{s_p}}-e^{-\theta t}-\sum_{n=1}^{M}\Psi_n\frac{b_p [S^n]}{2\sqrt{s_p}}, \qquad
				0=\frac{\gamma e^{-\theta t} b_q \left([S]-[S^2]\right)}{2\sqrt{s_q}}-e^{-\theta t}-\sum_{n=1}^{M}\Psi_n\frac{b_q \left([S^n]-[S^{n+1}]\right)}{2\sqrt{s_q}},
			\end{equation*}
			respectively, which lead to~\eqref{eq:s_lambda_pq_M}.

			\section{Proof of Theorem~\ref{lem:opt_complete_pq}}
			\label{app:lem:opt_complete_pq}
			By Theorem~\ref{thm:pontryagin}, any solution of~\eqref{eq:profit} must satisfy $\frac{\partial H}{\partial s_p}=\frac{\partial H}{\partial s_q}=0$, where
			\begin{subequations}
				\label{eqs:ham_pq}
				\begin{equation}
					\label{eq:ham_pq}
					H(t,f(t),s_p(t),s_q(t))=\Big(\gamma \frac{df}{dt}-s_p-s_q\Big)e^{-\theta t}+\Psi\frac{df}{dt},
				\end{equation}
				and
				\begin{equation}
					\label{eq:lambda_pq}
					\frac{d\Psi}{dt}=-\frac{\partial H}{\partial f}, \qquad \Psi(T)=0.
				\end{equation}
			\end{subequations}
			Substituting~\eqref{eq:complete_limit} in~\eqref{eqs:ham_pq} yields
			\begin{subequations}
				\label{eqs:H_lambda_pq}
				\begin{equation*}
					H=\Big(\gamma e^{-\theta t}+\Psi \Big)(1-f)\big(p(s_p)+q(s_q)f\big)-(s_p+s_q)e^{-\theta t},
				\end{equation*}
				and
				\begin{equation*}
					\frac{d\Psi}{dt}
					=\Big(\gamma e^{-\theta t}+\Psi\Big)\big(p(s_p)+q(s_q)(2f-1)\big), \qquad \Psi(T)=0.
				\end{equation*}
			\end{subequations}
			The condition $\frac{\partial H}{\partial s_p}=\frac{\partial H}{\partial s_q}=0$ follows from~\eqref{eq:s_condition}.
			
			\section{Proof of Corollary~\ref{cor:opt_complete_pq}}
			\label{app:cor:opt_complete_pq}
			Substituting~\eqref{eq:p_q_s} and~\eqref{eq:bvp_pq} in~\eqref{eq:H_inf}, the equation $\frac{\partial H}{\partial s_p}=\frac{\partial H}{\partial s_q}=0$ yields
			\begin{equation*}
				0=\frac{b_p(1-f)(\gamma e^{-\theta t}+\Psi)}{2\sqrt{s_p(t)}}-e^{-\theta t}, \qquad
				0=\frac{b_q(1-f)f(\gamma e^{-\theta t}+\Psi)}{2\sqrt{s_q(t)}}-e^{-\theta t},
			\end{equation*}
			which leads to~\eqref{eq:s_lambda_pq}.
			
			\section{Proof of Corollary~\ref{cor:sp_prime_zero}}
			\label{app:cor_sp_prime_zero}
			By~\eqref{eq:s_lambda_pq},
			\begin{equation}
				\label{eq:lambda_s_p}
				\Psi(t)=\Big(\frac{2\sqrt{s_p}}{b_p(1-f)}-\gamma\Big)e^{-\theta t}.
			\end{equation}
			Differentiating this equation and equating it to~$\frac{d\Psi}{dt}$ from~\eqref{eq:bvp_pq} yields
			\begin{align*}
				&\bigg(\frac{\frac{s_p'}{\sqrt{s_p}}b_p(1-f)+2\sqrt{s_p}b_p(1-f)\left(p_0+b_p\sqrt{s_p}+(q_0+b_q\sqrt{s_q})f\right)}{b_p^2(1-f)^2}
				-\theta\Big(\frac{2\sqrt{s_p}}{b_p(1-f)}-\gamma\Big)\bigg)e^{-\theta t}\\&\quad=\frac{2\sqrt{s_p}e^{-\theta t}}{b_p(1-f)}\Big(p_0+b_p\sqrt{s_p}+(q_0+b_q\sqrt{s_q})(2f-1)\Big).
			\end{align*}
			This equation simplifies to
			\begin{equation}
				\label{eq:s_p_prime}
				s_p'+2s_p\big((q_0+b_q\sqrt{s_q})(1-f)-\theta\big)+\sqrt{s_p}\theta\gamma b_p(1-f)=0.
			\end{equation}
			Assume by contradiction that $s_p(0)=0$. Then either there exists an interval $I=[0,t^*]$ for which $s_p(t)\equiv 0$, or $s_p(t)>0$ for $t>0$. By~\eqref{eq:lambda_s_p}, if $s_p(t)\equiv 0$, then $\Psi=-\gamma e^{-\theta t}$. Substituting this into~\eqref{eq:bvp_pq} gives $\frac{d\Psi}{dt}\equiv 0$, which contradicts with $\Psi=-\gamma e^{-\theta t}$. If $s_p(0)=0$ and $s_p(t)>0$ for $0<t\ll 1$, then $\left(q_0+b_q\sqrt{s_q(t)}\right)(1-f)-\theta>0$, and so~\eqref{eq:s_p_prime} implies that $s_p'(t)<0$, which is a contradiction. Therefore, we conclude that $s_p(0)>0$. By~\eqref{eq:s_p_prime} and Corollary~\ref{cor:s(0)},
			\begin{equation}
				\label{eq:s_p_prime_0}
				s_p'(0)+2s_p(0)\left(q_0-\theta\right)+\sqrt{s_p(0)}\theta\gamma b_p=0.
			\end{equation}
			Hence, $s_p'(0)<0$.

			\section{Proof of Corollary~\ref{cor:s_infty}}
			\label{app:cor:s_infty}
			We begin with an auxilliary lemma:
			
			\begin{lemma}
				\label{lem:lambda}
				Let $\tilde{t}=\left(f^{\rm opt}\right)^{-1}(\frac{1}{2})$. Then the solution of~\eqref{eq:bvp_pq} satisfies $-\gamma e^{-\theta t}\leq \Psi(t)\leq 0$ for  $\tilde{t}\leq t<T$.
			\end{lemma}
			\begin{proof}
				Since $s_p(t),s_q(t)\geq 0$, then
				\begin{equation*}
					p_0+b_p\sqrt{s_p(t)}+\Big(q_0+b_q\sqrt{s_q(t)}\Big)(2f-1)\geq p_0>0, \qquad \tilde{t}\leq t\leq T.
				\end{equation*}
				Therefore, by~\eqref{eq:bvp_pq}, the sign of $\frac{d\Psi}{dt}$ is determined by the sign of $A(t):=\gamma e^{-\theta t}+\Psi$. Let $\tilde{t} \leq t_1<T$.
				\begin{itemize}
					\item[{\bf Case 1:}]
					If $\Psi(t_1)>0$, then in order to have $\Psi(T)=0$, see~\eqref{eq:bvp_pq}, there must exist $t_2\in (t_1,T)$ such that $\Psi(t_2)>0$ and $\frac{d\Psi}{dt}(t_2)<0$. However, $\Psi(t_2)>0$ implies that $A(t_2)>0$, and hence that $\frac{d\Psi_1}{dt}(t_2)>0$. Contradiction.
					\item[{\bf Case 2:}]
					If $\Psi(t_1)<-\gamma e^{-\theta t_1}$, then $A(t_1)<0$. Since $A(T)=\gamma e^{-\theta T}>0$, there must exist $t_2\in (t_1,T)$ such that $A(t_2)=0$ and $A'(t_2)>0$. However,
					$
					A'(t)=-\gamma\theta e^{-\theta t}+\frac{d\Psi_1}{dt}.
					$
					Since $A(t_2)=0$, then $\frac{d\Psi}{dt}(t_2)=0$, see~\eqref{eq:bvp_pq}. Hence, $A'(t_2)=-\gamma\theta e^{-\theta t_2}<0$. Contradiction.
				\end{itemize}
				Therefore, the result holds
			\end{proof}
			
			\begin{proof}[Proof of Corollary~{\rm \ref{cor:s_infty}}:]
				Let $T<\infty$. Since $\Psi(t)$ and $f^{\rm opt}(t)$ are continuous, $s_p^{\rm opt}(t)$ and $s_q^{\rm opt}(t)$ are bounded in $[0,T]$, see~\eqref{eq:s_lambda_pq}. Therefore, by~\eqref{eq:bvp_pq},
				$
				\frac{df^{\rm opt}}{dt}\leq (1-f^{\rm opt})K,
				$
				where $K$ is a positive constant. Hence, $f^{\rm opt}(T)<1$, since $f$ cannot converge to the fixed point $f=1$ in finite time. Therefore, by~\eqref{eq:s_lambda_pq}, $s_p^{\rm opt}(T), s_q^{\rm opt}(T)>0$.
				
				Let $T=\infty$. Since $f$ is bounded from below by the case in which $q=b_p=b_q=0$, then $f(t)>1-e^{-pt}$. Hence, $\lim\limits_{t\to\infty}f^{\rm opt}(t)=1$. In addition, by~\eqref{eq:s_lambda_pq},
				$
				\label{eq:s_bounds}
				s_p(t)<\frac{b_p^2}{4}\left(e^{-pt}(\Psi e^{\theta t}+\gamma)\right)^2.
				$
				By Lemma~\ref{lem:lambda}, $0\leq \Psi e^{\theta t}+\gamma=\gamma$. Hence, $\lim\limits_{t\to\infty}s_p^{\rm opt}(t)=0$, and so by~\eqref{eq:s_lambda_pq}, $\lim\limits_{t\to\infty}s_q^{\rm opt}(t)=0$.
			\end{proof}

			\section{Proof of Theorem~\ref{lem:opt_circle_pq}}
			\label{app:lem:opt_circle_pq}
			
			Let $y:=\int_{0}^{t}p(s_p(\tau))\, d\tau$ be a new state variable. Then
			$\frac{dy}{dt}=p(s_p(t))$ and $y(0)=0$.
			Maximizing~\eqref{eq:profit} on the infinite line~\eqref{eq:circle_conditions} can be formulated as an optimal control problem with the two state variables ${\bf f}(t)=(f_1(t),f_2(t))=(f,y)$ and the two controls $s_p(t)$ and $s_q(t)$. Furthermore, by~\eqref{eq:circle_limit},
			\begin{align*}
				\frac{df}{dt}&=g_1:=\Big(p(s_p(t))+q(s_q(t))(1-e^{-y})\Big)(1-f),  &f(0)=0,\qquad\qquad\qquad\\
				\frac{dy}{dt}&=g_2:=p(s_p(t)),  &y(0)=0.\qquad\qquad\qquad
			\end{align*}
			Finally, $\pi=\big(\gamma\frac{df}{dt}-s_p-s_q\big)e^{-\theta t}$. Applying Theorem~\ref{thm:pontryagin} gives
			\begin{equation}
				\label{eq:ham_circle}
				H=\pi+\Psi_1g_1+\Psi_2g_2=(\gamma e^{-\theta t} +\Psi_1) (1-f) \left(p(s_p)+q(s_q)(1-e^{-y})\right)+p(s_p)\Psi_2-\left(s_p+s_q\right)e^{-\theta t},
			\end{equation}
			which gives~\eqref{eq:H_line}. Equations~\eqref{eq:bvp_pq_circle} for $\frac{d\Psi_1}{dt}$ and $\frac{d\Psi_2}{dt}$ follow from
			$
			\frac{d\Psi_1}{dt}=-\frac{\partial H}{\partial f}$ and $ \frac{d\Psi_2}{dt}=-\frac{\partial H}{\partial y}.
			$
			
			\section{Proof of Corollary~\ref{cor:opt_circle_pq}}
			Substituting~\eqref{eq:p_q_s}, the Hamiltonian~\eqref{eq:ham_circle} becomes
			\begin{align*}
				H=(\gamma e^{-\theta t} +\Psi_1) (1-f) \left(p_0+b_p\sqrt{s_p}+(q_0+b_q\sqrt{s_q})(1-e^{-y})\right)+\left(p_0+b_p\sqrt{s_p}\right)\Psi_2-\left(s_p+s_q\right)e^{-\theta t}.
			\end{align*}
			The necessary condition for optimality $\frac{\partial H}{\partial s_p}=\frac{\partial H}{\partial s_q}=0$ reads
			$$
			\frac{(\gamma e^{-\theta t}+\Psi_1)(1-f)b_p+\Psi_2b_p}{2\sqrt{s_p}}-e^{-\theta t}=\frac{(\gamma e^{-\theta t}+\Psi_1)(1-f)b_q(1-e^{-y})}{2\sqrt{s_q}}-e^{-\theta t}=0,
			$$
			from which the result follows.

			\section{Proof of Corollary~\ref{cor:s_infty_circle}}
			\label{app:cor:s_infty_circle}
			We begin with an auxilliary lemma:
			\begin{lemma}
				\label{lem:lambda_1d}
				The solution of~\eqref{eq:bvp_pq_circle} satisfies for $0\leq t\leq T$
				\begin{equation}
					\label{eq:psi_2_bound}
					-\gamma e^{-\theta t}\leq \Psi_1(t)\leq 0, \qquad
					\max \Bigl\{0, \Big(\frac{2}{b_p}\sqrt{s_p(t)}-(1-f)\gamma\Big)e^{-\theta t}\Bigr\}\leq \Psi_2(t)\leq \frac{2}{b_p}\sqrt{s_p(t)}e^{-\theta t}.
				\end{equation}
			\end{lemma}
			\begin{proof}
				Since $s_p(t),s_q(t)\geq 0$, then $y(t)\geq 0$, and so
				\begin{equation*}
					p_0+b_p\sqrt{s_p(t)}+\Big(q_0+b_q\sqrt{s_q(t)}\Big)(1-e^{-y})\geq p_0>0, \qquad 0\leq t\leq T.
				\end{equation*}
				Therefore, by~\eqref{eq:bvp_pq_circle}, the sign of $\frac{d\Psi_1}{dt}$ is determined by the sign of $A(t):=\gamma e^{-\theta t}+\Psi_1$. Let $0\leq t_1<T$.
				\begin{itemize}
					\item[{\bf Case 1:}]
					If $\Psi_1(t_1)>0$, then in order to have $\Psi_1(T)=0$, see~\eqref{eq:bvp_pq_circle}, there must exist $t_2\in (t_1,T)$ such that $\Psi_1(t_2)>0$ and $\frac{d\Psi_1}{dt}(t_2)<0$. Now, $\Psi_1(t_2)>0$ impliles that $A(t_2)>0$, and hence that $\frac{d\Psi_1}{dt}(t_2)>0$. Contradiction.
					\item[{\bf Case 2:}]
					If $\Psi_1(t_1)<-\gamma e^{-\theta t_1}$, then $A(t_1)<0$. Since $A(T)=\gamma e^{-\theta T}>0$, there must exist $t_2\in (t_1,T)$ such that $A(t_2)=0$ and $A'(t_2)>0$. Now,
					$
					A'(t)=-\gamma\theta e^{-\theta t}+\frac{d\Psi_1}{dt}.
					$
					Since $A(t_2)=0$, then $\frac{d\Psi_1}{dt}(t_2)=0$, see~\eqref{eq:bvp_pq_circle}. Hence, $A'(t_2)=-\gamma\theta e^{-\theta t_2}<0$. Contradiction.
				\end{itemize}
				Therefore, we proved that
				$$
				-\gamma e^{-\theta t}\leq\Psi_1(t)\leq 0, \qquad A(t)\geq 0 \qquad 0\leq t<T.
				$$
				Since
				$
				(1-f)\left(q_0+b_q\sqrt{s_q(t)}\right)e^{-y}>0
				$
				and $A(t)\geq 0$, then by~\eqref{eq:bvp_pq_circle}, $\frac{d\Psi_2}{dt}(t)\leq 0$ for $0\leq t<T$. Hence, since $\Psi_2(T)=0$, $\Psi_2(t)\geq 0$ for $0\leq t<T$. Furthermore, by~\eqref{eq:s_lambda_pq_circle},
				\begin{equation}
					\label{eq:lambda_2_circle}
					\Psi_2(t)=\Big(\frac{2}{b_p}\sqrt{s_p(t)}-(1-f)(\gamma+\Psi_1e^{\theta t})\Big)e^{-\theta t}.
				\end{equation}
				The upper and lower bounds for $\Psi_2$ are attained by substituting  $-\gamma e^{-\theta t}\leq\Psi_1 \leq 0$ in~\eqref{eq:lambda_2_circle}.
			\end{proof}
			
			\begin{proof}[Proof of Corollary~\ref{cor:s_infty_circle}]
				Let $T<\infty$. Since $\Psi_1(t), \Psi_2(t)$ and $f^{\rm opt}(t)$ are continuous, $s_p^{\rm opt}(t)$ and $s_q^{\rm opt}(t)$ are bounded in $[0,T]$, see~\eqref{eq:s_lambda_pq_circle}. Therefore, by~\eqref{eq:bvp_pq_circle},
				$
				\frac{df^{\rm opt}}{dt}\leq (1-f^{\rm opt})K,
				$
				where $K$ is a positive constant. Hence, $f^{\rm opt}(T)<1$. Therefore, by~\eqref{eq:s_lambda_pq_circle}, $s_p^{\rm opt}(T), s_q^{\rm opt}(T)>0$.
				
				Let $T=\infty$. By Lemma~\ref{lem:lambda_1d} and~\eqref{eq:s_lambda_pq_circle},
				$$
				s_p(t)\leq \frac{b_p^2}{4}\left(\gamma(1-f)+\Psi_2e^{\theta t}\right)^2,\qquad
				s_q(t)\leq \frac{b_q^2}{4}\Big(\gamma(1-f)(1-e^{-y})\Big)^2.
				$$
				
				Since $\lim\limits_{t\to\infty}f(t)=1$, then $\lim\limits_{t\to\infty}s_q(t)=0$.
				To show that $\lim\limits_{t\to\infty}s_p(t)=0$, we need to show that $\lim\limits_{t\to\infty}\Psi_2 e^{\theta t}=0$. Indeed, by L'Hospital's rule, and using~$\lim\limits_{t\to\infty}s_q^{\rm opt}(t)=0$,~\eqref{eq:bvp_pq_circle} and~\eqref{eq:psi_2_bound},
				\begin{align*}
					0\leq \lim\limits_{t\to\infty}\frac{\Psi_2(t)}{e^{-\theta t}}&=\lim\limits_{t\to\infty}\frac{(\gamma e^{-\theta t}+\Psi_1)(1-f)\left(q+b_q\sqrt{s_q(t)}\right)e^{-y}}{\theta e^{-\theta t}}\\&
					\leq \lim\limits_{t\to\infty}\frac{\gamma e^{-\theta t}(1-f)\left(q+b_q\sqrt{s_q(t)}\right)e^{-y}}{\theta e^{-\theta t}}=0.
				\end{align*}
			\end{proof}
			
			\section{Heterogeneous complete networks}
			\label{app:het}
			
			The reduction of the aggregate dynamics to two ordinary differential equations, see~\eqref{eq:groups_limit}, enables us to formulate the optimal promotion as a boundary-value problem with four equations:
			\begin{theorem}
				\label{thm:opt_group_same}
				Consider the Bass model~{\rm (\ref{eqs:Bass-model},~\ref{eq:het_conditions})} on a heterogeneous complete network with two homogeneous groups of equal size in which the effect of promotions is given by~\eqref{eq:p_q_s_het}. Then 
				\begin{subequations}
					\label{eqs:pq_opt_groups}
					\begin{equation}
						\label{eq:s_lambda_pq_groups}
						\begin{aligned}
							s_p^{\rm opt}(t)&=\frac{1}{4}\Big(\sum_{k=1}^{2}b_{p}^k\big(\frac{1}{2}-f_k\big)\big(\gamma +\Psi_ke^{\theta t}\big)\Big)^2, \quad
							s_q^{\rm opt}(t)=\frac{1}{4}\Big(\big(f_1+f_2\big)\sum_{k=1}^{2}b_{q}^k\big(\frac{1}{2}-f_k\big)\big(\gamma +\Psi_ke^{\theta t}\big)\Big)^2,
						\end{aligned}
					\end{equation}
					where $f_1^{\rm opt}, f_2^{\rm opt}, \Psi_1^{\rm opt},$ and $\Psi_2^{\rm opt}$ are the solutions of the boundary-value problem 
					\begin{equation}
						\label{eq:bvp_pq_groups}
						\begin{aligned}
							\frac{df_k}{dt}&=\Big(\frac{1}{2}-f_k\Big)R_k,\qquad &&f_k(0)=0, \qquad k=1,2,
							\\
							\frac{d\Psi_1}{dt}
							&=(\gamma e^{-\theta t}+\Psi_1) \Big(R_1-\Big(\frac{1}{2}-f_1\Big)\Big(q_0^1+b_{q}^1\sqrt{s_q^{\rm opt}}\Big)\Big)\\&\qquad\qquad\qquad-\left(\gamma e^{-\theta t}+\Psi_2\right)\Big(\frac{1}{2}-f_2\Big)\Big(q_0^2+b_{q}^2\sqrt{s_q^{\rm opt}}\Big), \qquad &&\Psi_1(T)=0,
							\\
							\frac{d\Psi_2}{dt}
							&=(\gamma e^{-\theta t}+\Psi_2) \Big(R_2-\Big(\frac{1}{2}-f_2\Big)\Big(q_0^2+b_{q}^2\sqrt{s_q^{\rm opt}}\Big)\Big)
							\\&\qquad\qquad\qquad-\Big(\gamma e^{-\theta t}+\Psi_1\Big)\Big(\frac{1}{2}-f_1\Big)\Big(q_0^1+b_{q}^1\sqrt{s_q^{\rm opt}}\Big),\qquad &&\Psi_2(T)=0,
						\end{aligned}
					\end{equation}
					and 
					\begin{equation}
						\label{eq:R_values}
						R_k:=p_0^k+b_{p^k}\sqrt{s_p^{\rm opt}}+\Big(q_0^k+b_{q}^k\sqrt{s_q^{\rm opt}}\Big)(f_1+f_2), \qquad k=1,2.
					\end{equation}
				\end{subequations}
			\end{theorem}
			\begin{proof}
				By Theorem~\ref{thm:pontryagin}, any solution of~\eqref{eq:profit} must satisfy $\frac{\partial H}{\partial s_p}=\frac{\partial H}{\partial s_q}=0$, where
				\begin{subequations}
					\label{eqs:ham_pq_same}
					\begin{equation}
						\label{eq:ham_pq_same}
						H(t)=\Big(\gamma \Big(\frac{df_1}{dt}+\frac{df_2}{dt}\Big)-s_p-s_q\Big)e^{-\theta t}+\Psi_1\frac{df_1}{dt}+\Psi_2\frac{df_2}{dt},
					\end{equation}
					and
					\begin{equation}
						\label{eq:lambda_pq_same}
						\frac{d\Psi_k}{dt}=-\frac{\partial H}{\partial f_k}, \qquad \Psi_k(T)=0, \qquad k=1,2.
					\end{equation}
				\end{subequations}
				Substituting~\eqref{eq:p_q_s_het}~and~\eqref{eq:groups_limit}~in~\eqref{eqs:ham_pq_same} yields
				\begin{subequations}
					\label{eqs:H_lambda_pq_same}
					\begin{equation*}
						\begin{aligned}
							H
							&=(\gamma e^{-\theta t} +\Psi_1) \Big(\frac{1}{2}-f_1\Big)\left(p_0^1+b_{p_1}\sqrt{s_p}+\left(q_0^1+b_{q_1}\sqrt{s_q}\right)(f_1+f_2)\right)\\ &\qquad+(\gamma e^{-\theta t} +\Psi_2) \Big(\frac{1}{2}-f_2\Big)\left(p_0^2+b_{p_2}\sqrt{s_p}+\left(q_0^2+b_{q_2}\sqrt{s_q}\right)(f_1+f_2)\right)-\left(s_p+s_q\right)e^{-\theta t},
						\end{aligned}		
					\end{equation*}
					and
					\begin{equation*}
						\label{eq:lambda_prime_pq_same}
						\begin{aligned}
							\frac{d\Psi_1}{dt}
							&=\left(\gamma e^{-\theta t}+\Psi_1\right)\left(p_0^1+b_{p_1}\sqrt{s_p}+\left(q_0^1+b_{q_1}\sqrt{s_q}\right)\left(2f_1+f_2-\frac{1}{2}\right)\right)\\ &\qquad\qquad-\left(\gamma e^{-\theta t}+\Psi_2\right)\left(q_0^2+b_{q_2}\sqrt{s_q}\right)\Big(\frac{1}{2}-f_2\Big), \qquad &\Psi_1(T)=0, \\
							\frac{d\Psi_2}{dt}
							&=\left(\gamma e^{-\theta t}+\Psi_2\right)\left(p_2+b_{p_2}\sqrt{s_p}+\left(q_2+b_{q_2}\sqrt{s_q}\right)\left(2f_2+f_1-\frac{1}{2}\right)\right)\\ &\qquad\qquad-\left(\gamma e^{-\theta t}+\Psi_1\right)\left(q_0^1+b_{q_1}\sqrt{s_q}\right)\Big(\frac{1}{2}-f_1\Big), \qquad &\Psi_2(T)=0.
						\end{aligned}
					\end{equation*}
					
				\end{subequations}
				Since $\frac{\partial H}{\partial s_p}=\frac{\partial H}{\partial s_q}=0$, we have
				\begin{align*}
					\label{eq:neccesary_pq_same}
					0&=\frac{b_{p_1}\Big(\frac{1}{2}-f_1\Big)\left(\gamma e^{-\theta t}+\Psi_1(t)\right)+b_{p_2}\Big(\frac{1}{2}-f_2\Big)\left(\gamma e^{-\theta t}+\Psi_2(t)\right)}{2\sqrt{s_p(t)}}-e^{-\theta t},\\
					0&=(f_1+f_2)\frac{b_{q_1}\Big(\frac{1}{2}-f_1\Big)\left(\gamma e^{-\theta t}+\Psi_1(t)\right)+b_{q_2}\Big(\frac{1}{2}-f_2\Big)\left(\gamma e^{-\theta t}+\Psi_2(t)\right)}{2\sqrt{s_q(t)}}-e^{-\theta t}.
				\end{align*}
				These equations lead to~\eqref{eq:s_lambda_pq_groups}.
			\end{proof}

			When we apply a heterogeneous promotion, see~\eqref{eq:p_q_s_diff}, we arrive at
			\begin{theorem}
				\label{thm:opt_group_diff}
				Consider the Bass model~{\rm (\ref{eqs:Bass-model},~\ref{eq:het_conditions})} on a heterogeneous complete network with two homogeneous groups of equal size in which the effect of promotions is given by~\eqref{eq:p_q_s}. Then
				\begin{subequations}
					\label{eqs:pq_opt_groups_diff}
					\begin{equation}
						\label{eq:s_lambda_pq_groups_diff}
						\begin{aligned}
							s_{p}^{k,\rm opt}(t)&=\Big(b_{p^k}\left(\frac{1}{2}-f_k\right)\Big(\gamma +\Psi_ke^{\theta t}\Big)\Big)^2, \qquad&k=1,2,
							\\
							s_{q}^{k,\rm opt}(t)&=\Big((f_1+f_2)b_{q}^k\Big(\frac{1}{2}-f_k\Big)\Big(\gamma +\Psi_k e^{\theta t}\Big)\Big)^2,\qquad &k=1,2,
						\end{aligned}
					\end{equation}
					where $f_1^{\rm opt}, f_2^{\rm opt}, \Psi_1^{\rm opt},$ and $\Psi_2^{\rm opt}$ are the solutions of the boundary-value problem 
					\begin{equation}
						\label{eq:bvp_pq_groups_diff}
						\begin{aligned}
							\frac{df_k}{dt}&=\Big(\frac{1}{2}-f_k\Big)R_k,\qquad &&f_k(0)=0,\quad k=1,2,\\
							\frac{d\Psi_1}{dt}
							&=(\gamma e^{-\theta t}+\Psi_1) \Big(R_1-\Big(\frac{1}{2}-f_1\Big)\left(q_0^1+b_{q}^1\sqrt{s_{q}^{1,\rm opt}}\right)\Big)\\&\qquad\qquad\qquad-\left(\gamma e^{-\theta t}+\Psi_2\right)\Big(\frac{1}{2}-f_2\Big)\left(q_0^2+b_{q}^2\sqrt{s_{q}^{2,\rm opt}}\right), \qquad &&\Psi_1(T)=0,
							\\
							\frac{d\Psi_2}{dt}
							&=(\gamma e^{-\theta t}+\Psi_2) \Big(R_2-\Big(\frac{1}{2}-f_2\Big)\Big(q_0^2+b_{q}^2\sqrt{s_{q}^{2,\rm opt}}\Big)\Big)\\&\qquad\qquad\qquad -\left(\gamma e^{-\theta t}+\Psi_1\right)\Big(\frac{1}{2}-f_1\Big)\Big(q_0^1+b_{q}^1\sqrt{s_{q}^{1,\rm opt}}\Big),\qquad &&\Psi_2(T)=0,
						\end{aligned}
					\end{equation}
					and 
					\begin{equation}
						\label{eq:R_values_diff}
						R_k:=p_0^k+b_{p^k}\sqrt{s_{p}^{k,\rm opt}}+\left(q_0^k+b_{q}^k\sqrt{s_{q}^{k,\rm opt}}\right)(f_1+f_2), \qquad k=1,2.
					\end{equation}
				\end{subequations}
			\end{theorem}
			\begin{proof}
				By Theorem~\ref{thm:pontryagin}, any solution of~\eqref{eq:profit} must satisfy $\frac{\partial H}{\partial s_p^k}=\frac{\partial H}{\partial s_q^k}=0$, where
				\begin{subequations}
					\label{eqs:ham_pq_diff}
					\begin{equation}
						\label{eq:ham_pq_diff}
						H=\Big(\gamma \Big(\frac{df_1}{dt}+\frac{df_2}{dt}\Big)-\frac{1}{2}(s_{p_1}+s_{p_2}+s_{q_1}+s_{q_2})\Big)e^{-\theta t}+\Psi_1\frac{df_1}{dt}+\Psi_2\frac{df_2}{dt},
					\end{equation}
					and
					\begin{equation}
						\label{eq:lambda_pq_diff}
						\frac{d\Psi_k}{dt}=-\frac{\partial H}{\partial f_k}, \qquad \Psi_k(T)=0, \qquad k=1,2.
					\end{equation}
				\end{subequations}
				Substituting~\eqref{eq:groups_limit} with~\eqref{eq:p_q_s_diff} in~\eqref{eqs:ham_pq_diff} yields
				\begin{subequations}
					\label{eqs:H_lambda_pq_diff}
					\begin{equation*}
						\begin{aligned}
							H(t)
							&=(\gamma e^{-\theta t} +\Psi_1) (\frac{1}{2}-f_1)\Big(p_0^1+b_{p_1}\sqrt{s_{p_1}(t)}+\Big(q_0^1+b_{q_1}\sqrt{s_{q_1}(t)}\Big)(f_1+f_2)\Big)\\ &\qquad+(\gamma e^{-\theta t} +\Psi_2) (\frac{1}{2}-f_2)\Big(p_0^2+b_{p_2}\sqrt{s_{p_2}(t)}+\Big(q_0^2+b_{q_2}\sqrt{s_{q_2}(t)}\Big)(f_1+f_2)\Big)\\ &\qquad-\frac{1}{2}\left(s_{p_1}(t)+s_{p_2}(t)+s_{q_1}(t)+s_{q_2}(t)\right)e^{-\theta t},
						\end{aligned}		
					\end{equation*}
					and
					\begin{equation*}
						\label{eq:lambda_prime_pq_diff}
						\begin{aligned}
							\frac{d\Psi_1}{dt}
							&=\Big(\gamma e^{-\theta t}+\Psi_1(t)\Big)\Big(p_0^1+b_{p_1}\sqrt{s_{p_1}(t)}+\Big(q_0^1+b_{q_1}\sqrt{s_{q_1}(t)}\Big)\Big(2f_1+f_2-\frac{1}{2}\Big)\Big)\\ &\qquad\qquad-\Big(\gamma e^{-\theta t}+\Psi_2(t)\Big)\Big(q_0^2+b_{q_2}\sqrt{s_{q_2}(t)}\Big)\Big(\frac{1}{2}-f_2\Big), \qquad &\Psi_1(T)=0, \\
							\frac{d\Psi_2}{dt}
							&=\Big(\gamma e^{-\theta t}+\Psi_2(t)\Big)\Big(p_0^2+b_{p_2}\sqrt{s_{p_2}(t)}+\Big(q_0^2+b_{q_2}\sqrt{s_{q_2}(t)}\Big)\Big(2f_2+f_1-\frac{1}{2}\Big)\Big)\\ &\qquad\qquad-\left(\gamma e^{-\theta t}+\Psi_1(t)\right)\Big(q_0^1+b_{q_1}\sqrt{s_{q_1}(t)}\Big)\Big(\frac{1}{2}-f_1\Big), \qquad &\Psi_2(T)=0.
						\end{aligned}
					\end{equation*}
					
				\end{subequations}
				Since $\frac{\partial H}{\partial s_{p_1}}=\frac{\partial H}{\partial s_{p_2}}=\frac{\partial H}{\partial s_{q_1}}=\frac{\partial H}{\partial s_{q_2}}=0$, we have for $k=1,2$,
				\begin{equation*}
					\label{eq:neccesary_pq_diff}
					0=\frac{b_{p_k}\left(\frac{1}{2}-f_k\right)\left(\gamma e^{-\theta t}+\Psi_k(t)\right)}{2\sqrt{s_{p_k}(t)}}-\frac{1}{2}e^{-\theta t},\qquad
					0=(f_1+f_2)\frac{b_{q_k}\left(\frac{1}{2}-f_k\right)\left(\gamma e^{-\theta t}+\Psi_k(t)\right)}{2\sqrt{s_{q_k}(t)}}-\frac{1}{2}e^{-\theta t}.
				\end{equation*}
				These equations lead to~\eqref{eq:s_lambda_pq_groups_diff}.
			\end{proof}
			
			When we apply a heterogeneous promotion to an innovator-imitator population, see~\eqref{eq:innovator}, we arrive at
			
			\begin{theorem}
				\label{thm:opt_influencer}
				Consider the Bass model~{\rm (\ref{eqs:Bass-model},~\ref{eq:innovator})} on a heterogeneous complete network with two homogeneous groups of equal size in which the effect of promotions is given by~\eqref{eq:p_q_s}. Then
				\begin{subequations}
					\label{eqs:pq_opt_groups_inn}
					\begin{equation}
						\label{eq:s_lambda_pq_groups_inn}
						\begin{aligned}
							s_{p_1}^{\rm opt}(t)&=\Big(b_{p_1}\left(\frac{1}{2}-f_1\right)\Big(\gamma +\Psi_1e^{\theta t}\Big)+b_p^{1\to 2}\left(\frac{1}{2}-f_2\right)\Big(\gamma +\Psi_2e^{\theta t}\Big)\Big)^2, 
							\\
							s_{q_1}^{\rm opt}(t)&=\Big(\left(f_1+f_2\right)\Big(b_{q_1}\left(\frac{1}{2}-f_1\right)\Big(\gamma +\Psi_1e^{\theta t}\Big)+b_q^{1\to 2}\left(\frac{1}{2}-f_2\right)\Big(\gamma +\Psi_2e^{\theta t}\Big)\Big)\Big)^2, \\
							s_{p_2}^{\rm opt}(t)&=\Big(b_{p_2}\Big(\frac{1}{2}-f_2\Big)\Big(\gamma +\Psi_2 e^{\theta t}\Big)\Big)^2,\\
							s_{q_2}^{\rm opt}(t)&=\Big((f_1+f_2)b_{q_2}\Big(\frac{1}{2}-f_2\Big)\Big(\gamma +\Psi_2 e^{\theta t}\Big)\Big)^2.
						\end{aligned}
					\end{equation}
					where $f_1^{\rm opt}, f_2^{\rm opt}, \Psi_1^{\rm opt},$ and $\Psi_2^{\rm opt}$ are the solutions of the boundary-value problem 
					\begin{equation}
						\label{eq:bvp_pq_groups_inn}
						\begin{aligned}
							\frac{df_k}{dt}&=\Big(\frac{1}{2}-f_k\Big)R_k, &f_k(0)=0,\quad k=1,2,
							\\
							\frac{d\Psi_1}{dt}
							&=(\gamma e^{-\theta t}+\Psi_1) \Big(R_1-\Big(\frac{1}{2}-f_1\Big)\left(q_0^1+b_{q}^1\sqrt{s_{q}^{1,\rm opt}}\right)\Big) 
							\\
							&\qquad-\left(\gamma e^{-\theta t}+\Psi_2\right)\Big(\frac{1}{2}-f_2\Big)\left(q_0^2+b_{q}^2\sqrt{s_{q}^{2,\rm opt}}+b_q^{1\to 2}\sqrt{s_{q}^{1,\rm opt}}\right),  &\Psi_1(T)=0,
							\\
							\frac{d\Psi_2}{dt} & = (\gamma e^{-\theta t}+\Psi_2) \Big(R_2-\Big(\frac{1}{2}-f_2\Big)\Big(q_0^2+b_{q}^2\sqrt{s_{q}^{2,\rm opt}}+b_q^{1\to 2}\sqrt{s_{q}^{1,\rm opt}}\Big)\Big)
							\\
							& \qquad -\left(\gamma e^{-\theta t}+\Psi_1\right)\Big(\frac{1}{2}-f_1\Big)\Big(q_0^1+b_{q}^1\sqrt{s_{q}^{1,\rm opt}}\Big), &\Psi_2(T)=0,
						\end{aligned}
					\end{equation}
					and 
					\begin{equation}
						\label{eq:R_values_inn}
						\begin{aligned}
							R_1&:=p_0^1+b_{p}^1\sqrt{s_{p}^{1,\rm opt}}+\left(q_0^1+b_{q}^1\sqrt{s_{q}^{1,\rm opt}}\right)(f_1+f_2),\\
							R_2&:=p_0^2+b_{p}^2\sqrt{s_{p}^{2, \rm opt}}+c_{p}\sqrt{s_{p}^{1,\rm opt}}+\left(q_0^2+b_{q}^2\sqrt{s_{q}^{1,\rm opt}}+c_{q}\sqrt{s_{q}^{1,\rm opt}}\right)(f_1+f_2),\\
						\end{aligned}
					\end{equation}
				\end{subequations}
			\end{theorem}
			\begin{proof}
				By Theorem~\ref{thm:pontryagin}, any solution of~\eqref{eq:profit} must satisfy $\frac{\partial H}{\partial s_p^k}=\frac{\partial H}{\partial s_q^k}=0$, where
				\begin{subequations}
					\label{eqs:ham_pq_inn}
					\begin{equation}
						\label{eq:ham_pq_inn}
						H=\Big(\gamma \Big(\frac{df_1}{dt}+\frac{df_2}{dt}\Big)-\frac{1}{2}(s_{p_1}+s_{p_2}+s_{q_1}+s_{q_2})\Big)e^{-\theta t}+\Psi_1\frac{df_1}{dt}+\Psi_2\frac{df_2}{dt},
					\end{equation}
					and
					\begin{equation}
						\label{eq:lambda_pq_inn}
						\frac{d\Psi_k}{dt}=-\frac{\partial H}{\partial f_k}, \qquad \Psi_k(T)=0, \qquad k=1,2.
					\end{equation}
				\end{subequations}
				Substituting~\eqref{eq:groups_limit} with~\eqref{eq:innovator} in~\eqref{eqs:ham_pq_inn} yields
				\begin{subequations}
					\label{eqs:H_lambda_pq_inn}
					\begin{equation*}
						\begin{aligned}
							H(t)
							&=(\gamma e^{-\theta t} +\Psi_1) (\frac{1}{2}-f_1)\Big(p_0^1+b_{p_1}\sqrt{s_{p_1}(t)}+\Big(q_0^1+b_{q_1}\sqrt{s_{q_1}(t)}\Big)(f_1+f_2)\Big)\\ &+(\gamma e^{-\theta t} +\Psi_2) (\frac{1}{2}-f_2)\Big(p_0^2+b_{p_2}\sqrt{s_{p_2}(t)}+b_p^{1\to 2}\sqrt{s_{p_1}(t)}+\Big(q_0^2+b_{q_2}\sqrt{s_{q_2}(t)}+c_{q}\sqrt{s_{q_1}(t)}\Big)(f_1+f_2)\Big)\\ &\qquad-\frac{1}{2}\left(s_{p_1}(t)+s_{p_2}(t)+s_{q_1}(t)+s_{q_2}(t)\right)e^{-\theta t},
						\end{aligned}		
					\end{equation*}
					and
					\begin{equation*}
						\label{eq:lambda_prime_pq_inn}
						\begin{aligned}
							\frac{d\Psi_1}{dt}
							&=\Big(\gamma e^{-\theta t}+\Psi_1(t)\Big)\Big(p_0^1+b_{p_1}\sqrt{s_{p_1}(t)}+\Big(q_0^1+b_{q_1}\sqrt{s_{q_1}(t)}\Big)\Big(2f_1+f_2-\frac{1}{2}\Big)\Big)\\ &\qquad\qquad-\Big(\gamma e^{-\theta t}+\Psi_2(t)\Big)\Big(q_0^2+b_{q_2}\sqrt{s_{q_2}(t)}+c_{q}\sqrt{s_{q_1}(t)}\Big)\Big(\frac{1}{2}-f_2\Big), \qquad\ \  \Psi_1(T)=0, \\
							\frac{d\Psi_2}{dt}
							&=\Big(\gamma e^{-\theta t}+\Psi_2(t)\Big)\Big(p_0^2+b_{p_2}\sqrt{s_{p_2}(t)}+b_p^{1\to 2}\sqrt{s_{p_1}(t)}+\Big(q_0^2+b_{q_2}\sqrt{s_{q_2}(t)}+c_{q}\sqrt{s_{q_1}(t)}\Big)\Big(2f_2+f_1-\frac{1}{2}\Big)\Big)\\ &\qquad\qquad-\left(\gamma e^{-\theta t}+\Psi_1(t)\right)\Big(q_0^1+b_{q_1}\sqrt{s_{q_1}(t)}\Big)\Big(\frac{1}{2}-f_1\Big), \qquad \qquad \qquad \qquad \Psi_2(T)=0.
						\end{aligned}
					\end{equation*}
					
				\end{subequations}
				Since $\frac{\partial H}{\partial s_{p_1}}=\frac{\partial H}{\partial s_{p_2}}=\frac{\partial H}{\partial s_{q_1}}=\frac{\partial H}{\partial s_{q_2}}=0$, we have
				\begin{align*}
					0&=\frac{b_{p_1}\left(\frac{1}{2}-f_1\right)\left(\gamma e^{-\theta t}+\Psi_1(t)\right)}{2\sqrt{s_{p_1}(t)}}+\frac{b_p^{1\to 2}\left(\frac{1}{2}-f_2\right)\left(\gamma e^{-\theta t}+\Psi_2(t)\right)}{2\sqrt{s_{p_1}(t)}}-\frac{1}{2}e^{-\theta t},\\
					0&=(f_1+f_2)\frac{b_{q_1}\left(\frac{1}{2}-f_1\right)\left(\gamma e^{-\theta t}+\Psi_1(t)\right)}{2\sqrt{s_{q_1}(t)}}+(f_1+f_2)\frac{b_q^{1\to 2}\left(\frac{1}{2}-f_2\right)\left(\gamma e^{-\theta t}+\Psi_2(t)\right)}{2\sqrt{s_{q_1}(t)}}-\frac{1}{2}e^{-\theta t},\\
					0&=\frac{b_{p_2}\left(\frac{1}{2}-f_2\right)\left(\gamma e^{-\theta t}+\Psi_2(t)\right)}{2\sqrt{s_{p_2}(t)}}-\frac{1}{2}e^{-\theta t},\\
					0&=(f_1+f_2)\frac{b_{q_2}\left(\frac{1}{2}-f_2\right)\left(\gamma e^{-\theta t}+\Psi_2(t)\right)}{2\sqrt{s_{q_2}(t)}}-\frac{1}{2}e^{-\theta t}.
				\end{align*}
				These equations lead to~\eqref{eq:s_lambda_pq_groups_inn}.
			\end{proof}


			\section{Proof of Theorem~\ref{thm:non_uniform}}
			\label{app:thm:non_uniform}
			
			Set $z(t) := f_1(t)+f_2(t)$. By concavity of the square root,
			\begin{equation}
				\label{eq:concavity-square-root}
				s_p(t)+ s_q(t)
				=\Bigg(\frac{\sqrt{s_p^1(t)}+\sqrt{s_p^2(t)}}{2}\Bigg)^2
				+\Bigg(\frac{\sqrt{s_q^1(t)}+\sqrt{s_q^2(t)}}{2}\Bigg)^2
				\le \frac{s_p^1(t)+s_p^2(t)}{2}+ \frac{s_q^1(t)+s_q^2(t)}{2}.
			\end{equation}
			Hence, since $\gamma\ge 0$,
			\begin{align*}
				\Pi[s_p^1,s_p^2,s_q^1,s_q^2]
				&= \gamma \int_0^T e^{-\theta t} z'(t)\,dt
				- \int_0^T e^{-\theta t}
				\left(\frac{s_p^1(t)+s_p^2(t)}{2} + \frac{s_q^1(t)+s_q^2(t)}{2}\right)dt \\
				&\le \gamma \int_0^T e^{-\theta t} z'(t)\,dt
				- \int_0^T e^{-\theta t}\bigl(s_p(t)+s_q(t)\bigr)dt.
			\end{align*}
			It thus suffices to prove
			\begin{equation}
				\label{eq:int_z'<=int_f'}
				\int_0^T e^{-\theta t} z'(t)\,dt
				\le
				\int_0^T e^{-\theta t} f'(t)\,dt.
			\end{equation}
			We will prove the stronger pointwise comparison
			$z(t)\le f(t)$ for all $t\in[0,T]$. To this end, define 
			\[
			x_i := \sqrt{s_p^i},
			\qquad
			y_i := \sqrt{s_q^i},
			\qquad
			\bar x := \frac{x_1+x_2}{2}=\sqrt{s_p},
			\qquad
			\bar y := \frac{y_1+y_2}{2}=\sqrt{s_q},
			\]
			\[
			D := f_1-f_2,
			\qquad
			A := p_0+b_p\bar x,
			\qquad
			B := q_0+b_q\bar y,
			\qquad
			C := b_p(x_1-x_2) + b_q z(y_1-y_2).
			\]
			A direct calculation from the two subgroup equations~\eqref{eq:f_k-proof} gives
			\begin{equation}
				z' = (1-z)(A+Bz) - \frac{1}{2}CD, \qquad 
				D' = \frac{1}{2}(1-z)C - (A+Bz)D. \label{eq:zDprime}
			\end{equation}
			
			We first note that $f_i(t)<1/2$ on $[0,S]$ for each $i=1,2$. Indeed, each quantity $g_i := \frac{1}{2}-f_i$ solves the linear equation
			\[
			g_i' = -g_i\bigl(p_0+b_p x_i+(q_0+b_q y_i)z\bigr),
			\qquad g_i(0)=\frac{1}{2},
			\]
			so $g_i(t)=\frac{1}{2}\exp\!\left(-\int_0^t \bigl(p_0+b_p x_i(\tau)+(q_0+b_q y_i(\tau))z(\tau)\bigr)\,d\tau\right)>0$. Hence $0\le f_i(t)<1/2$, which implies
			\[
			|D(t)| \le \Big|\frac12-f_1\Big|+\Big|\frac12-f_2\Big| = 1-z(t), \qquad  t\in[0,T].
			\]
			In particular, the normalized imbalance \(v := \frac{D}{1-z}\) is well defined and satisfies $|v|<1$ on $(0,T]$.
			
			Using~\eqref{eq:zDprime}, we compute
			\[
			v' = \frac{D'(1-z)+Dz'}{(1-z)^2} = \frac{C}{2}(1-v^2).
			\]
			Now define \(L := -\frac{1}{2}\log(1-v^2)\). Then $L\ge 0$, $L(0)=0$, and
			\[
			L' = \frac{vv'}{1-v^2} = \frac{vC}{2}.
			\]
			Since $D=(1-z)v$, equation \eqref{eq:zDprime} for $z'$ becomes
			\[
			z'=(1-z)\bigl(A+Bz-L'\bigr).
			\]
			Therefore, if we set
			\(w := \frac{e^L}{1-z}\),
			and recall that $B\ge 0$ and $e^L\ge 1$, we have 
			\begin{align*}
				w' = \frac{e^L L'}{1-z} + \frac{e^L z'}{(1-z)^2} = \frac{e^L}{1-z}\bigl(L' + A+Bz-L'\bigr) = (A+B)w - Be^L \le (A+B)w - B.
			\end{align*}

			For the uniform-promotion trajectory, define \(W := \frac{1}{1-f}\). Since $f'=(1-f)(A+Bf)$ and $f(0)=0$, one has 
			\[
			W'=(A+B)W-B,
			\qquad W(0)=1,
			\qquad w(0)=1.
			\]
			Hence
			\[
			\bigl(e^{-\int_0^t (A+B)}(W-w)\bigr)' = e^{-\int_0^t (A+B)}\bigl((W-w)'-(A+B)(W-w)\bigr)\ge 0,
			\]
			Since $W(0)-w(0)=0$, it follows that
			\[
			w(t)\le W(t), \qquad  t\in[0,T].
			\]
			Finally, because $L\ge 0$,
			\[
			\frac{1}{1-z}=we^{-L}\le w\le W=\frac{1}{1-f}.
			\]
			Therefore $1-z\ge 1-f$, and so
			\[
			z(t)\le f(t), \qquad  t\in[0,T].
			\]
			This proves the pointwise comparison. Since $z(0)=f(0)=0$, integration by parts yields
			\[
			\int_0^T e^{-\theta t} z'(t)\,dt
			= e^{-\theta T}z(T) + \theta \int_0^T e^{-\theta t} z(t)\,dt
			\le e^{-\theta T}f(T) + \theta \int_0^T e^{-\theta t} f(t)\,dt
			= \int_0^T e^{-\theta t} f'(t)\,dt,
			\]
			which establishes~\eqref{eq:int_z'<=int_f'}. Finally, an equality in~\eqref{eq:concavity-square-root} holds if and only if $s_p^1(t) = s_p^2(t)$ and $s_q^1(t) = s_q^2(t)$ for all $t \in [0, T]$.

			\section{Proof of Corollary~\ref{cor:opt_complete_1/n}}
			\label{app:cor:opt_complete_1/n}
			
			Substituting~\eqref{eq:p_q_s_1/n} and~\eqref{eq:bvp_pq} into~\eqref{eq:H_inf}, the first-order conditions $\frac{\partial H}{\partial s_p} = \frac{\partial H}{\partial s_q} = 0$ yield
			\begin{equation*}
				0=\frac{b_p(1-f)(\gamma e^{-\theta t}+\Psi)(s_p)^{\frac{1}{k}-1}}{k}-e^{-\theta t}, \qquad
				0=\frac{b_q(1-f)f(\gamma e^{-\theta t}+\Psi)(s_q)^{\frac{1}{k}-1}}{n}-e^{-\theta t},
			\end{equation*}
			which implies~\eqref{eq:s_lambda_1/n}. Similarly, substituting~\eqref{eq:p_q_s_ln} and~\eqref{eq:bvp_pq} into~\eqref{eq:H_inf}, the same first-order conditions yield
			\begin{equation*}
				0 = \frac{b_p(1-f)(\gamma e^{-\theta t}+\Psi)}{1+s_p}-e^{-\theta t}, \qquad
				0 = \frac{b_q(1-f)f(\gamma e^{-\theta t}+\Psi)}{1+s_q}-e^{-\theta t},
			\end{equation*}
			from which it follows that
			\begin{equation*}
				s_p(t) = b_p(1-f)(\Psi e^{\theta t}+\gamma) - 1, \qquad s_q(t) = b_q(1-f)f(\Psi e^{\theta t}+\gamma) - 1.
			\end{equation*}
			By Lemma~\ref{lem:lambda}, whose proof carries over unchanged when replacing the square root with a logarithm, we have $\Psi e^{\theta t} + \gamma \geq 0$ for $f > \tfrac{1}{2}$. Consequently, as $f \to 1$, both $s_p$ and $s_q$ become negative, violating the feasibility condition $s_p, s_q \geq 0$. In this case, the maximum is attained on the boundary of ${\cal S}^m$, where $s_p=0$  or $s_q = 0$.

			\section{Proof of Corollary~\ref{cor:opt_complete_M_1/n}}
			\label{app:cor:opt_complete_M_1/n}
			
			Substituting~\eqref{eqs:pq_opt_M} and~\eqref{eq:p_q_s_1/n} into~\eqref{eq:H_comp}, the first-order conditions $\frac{\partial H}{\partial s_p}=0$ and $\frac{\partial H}{\partial s_q}=0$ yield
			\begin{align*}
				0&=\frac{\gamma e^{-\theta t} b_p [S]}{k}s_p^{\frac{1}{k}-1}-e^{-\theta t}-\sum_{n=1}^{M}\Psi_n\frac{b_p [S^n]}{k}s_p^{\frac{1}{k}-1}, \\
				0&=\frac{\gamma e^{-\theta t} b_q \left([S]-[S^2]\right)}{k}s_q^{\frac{1}{k}-1}-e^{-\theta t}-\sum_{n=1}^{M}\Psi_n\frac{b_q \left([S^n]-[S^{n+1}]\right)}{k}s_q^{\frac{1}{k}-1},
			\end{align*}
			which imply~\eqref{eq:s_lambda_pq_M_1/n}. Similarly, substituting~\eqref{eqs:pq_opt_M} and~\eqref{eq:p_q_s_ln} into~\eqref{eq:H_comp}, the same first-order conditions yield
			\begin{equation*}
				0=\frac{\gamma e^{-\theta t} b_p [S]}{1+s_p}-e^{-\theta t}-\sum_{n=1}^{M}\Psi_n\frac{b_p [S^n]}{1+s_p}, \qquad
				0=\frac{\gamma e^{-\theta t} b_q \left([S]-[S^2]\right)}{1+s_q}-e^{-\theta t}-\sum_{n=1}^{M}\Psi_n\frac{b_q \left([S^n]-[S^{n+1}]\right)}{1+s_q},
			\end{equation*}
			which lead to~\eqref{eq:s_lambda_pq_M_ln}.

			\section{Proof of Theorem~\ref{thm:opt_complete_scale2}}
			\label{app:opt_complete_scale2}	
			
			By Theorem~\ref{thm:pontryagin}, any solution of~\eqref{eq:profit} must satisfy $\frac{\partial H}{\partial s_p}=\frac{\partial H}{\partial s_q}=0$, where
			\begin{subequations}
				\label{eqs:ham_prop}
				\begin{equation}
					\label{eq:ham_prop}
					H(t,f(t),s_p(t),s_q(t))=\Big(\gamma \frac{df}{dt}-(1-f)s_p-f(1-f)s_q\Big)e^{-\theta t}+\Psi\frac{df}{dt},
				\end{equation}
				and
				\begin{equation}
					\label{eq:lambda_prop}
					\frac{d\Psi}{dt}=-\frac{\partial H}{\partial f}, \qquad \Psi(T)=0.
				\end{equation}
			\end{subequations}
			Substituting~\eqref{eq:complete_limit} in~\eqref{eqs:ham_prop} yields
			\begin{subequations}
				\label{eqs:H_lambda_prop}
				\begin{equation*}
					H=\Big(\gamma e^{-\theta t}+\Psi \Big)(1-f)\big(p(s_p)+q(s_q)f\big)-((1-f)s_p+f(1-f)s_q)e^{-\theta t},
				\end{equation*}
				and
				\begin{equation*}
					\frac{d\Psi}{dt}
					=\Big(\gamma e^{-\theta t}+\Psi\Big)\big(p(s_p)+q(s_q)(2f-1)\big), \qquad \Psi(T)=0.
				\end{equation*}
			\end{subequations}
			Equation~\eqref{eq:s_lambda_pq_scale2} follows from substituting~\eqref{eq:p_q_s} into~\eqref{eqs:H_lambda_prop} and taking the condition $\frac{\partial H}{\partial s_p}=\frac{\partial H}{\partial s_q}=0$ from~\eqref{eq:s_condition}.
			
			\section{Numerical Methods}
			\label{app:numerical}
			In this paper, we solve boundary-value problems of the form
			\begin{subequations}
				\label{eqs:bvp_gen}
				\begin{align}
					\label{eq:f_gen}
					\frac{d{\bf f}}{dt}&={\bf g}\Big(t,{\bf f}(t),{\bf s}\big(t,{\bf f}(t),\boldsymbol{\Psi}(t)\big)\Big), &&{\bf f}(0)={\bf f}_0,\\
					\label{eq:lambda_gen}
					\frac{d\boldsymbol{\Psi}}{dt}&={\bf h}\Big(t,{\bf f}(t),{\bf s}\big(t,{\bf f}(t),\boldsymbol{\Psi}(t)\big),\boldsymbol{\Psi}(t)\Big), &&\boldsymbol{\Psi}(T)=0,
				\end{align}
			\end{subequations}
			where ${\bf f}, \boldsymbol{\Psi}\in \mathbb{R}\to\mathbb{R}^d$. This is a system of $2d$ nonlinear ODEs with $d$ boundary conditions for ${\bf f}$ at $t=0$ and $d$ boundary conditions for $\boldsymbol{\Psi}$ at $t=T$. We use two numerical methods to solve~\eqref{eqs:bvp_gen}.
			
			In the {\bf forward-backward sweeps method}, we freeze the value of $\boldsymbol{\Psi}(t)$ while solving~\eqref{eq:f_gen} from left to right, i.e., for increasing values of $t$, and then freeze the value of ${\bf f}(t)$ while solving~\eqref{eq:lambda_gen} from right to left. The algorithm proceeds as follows:
			\begin{enumerate}
				\item 
				Set $\boldsymbol{\Psi}^{(0)}(t)\equiv 0$.
				\item 
				Solve the initial-value problem~\eqref{eq:f_gen} in $0\leq t\leq T$ for ${\bf f}^{(n)}(t)$ using $\boldsymbol{\Psi}(t)=\boldsymbol{\Psi}^{(n-1)}(t)$.
				\item 
				Solve the intial-value problem~\eqref{eq:lambda_gen} in $T\geq t\geq 0$ for $\boldsymbol{\Psi}^{(n)}(t)$ using ${\bf f}(t)={\bf f}^{(n)}(t)$ .
				\item 
				Repeat steps 2-3 until $\|{\bf f}^{(n)}-{\bf f}^{(n-1)}\|_{\infty}<\text{TOL}$
			\end{enumerate}

			\noindent In the {\bf shooting method}, we proceed as follows:
			\begin{enumerate}
				\item 
				Guess an initial value $\boldsymbol{\Psi}_0$.
				\item 
				Solve~\eqref{eqs:bvp_gen} for $0\leq t\leq T$ as an initial value problem with $\boldsymbol{\Psi}(0)=\boldsymbol{\Psi}_0$.
				\item 
				Use a root-finding method to search for $\boldsymbol{\Psi}_0$ for which $\boldsymbol{\Psi}(T)=0$.
			\end{enumerate}
			
			In most cases, see e.g.,~\eqref{eq:bvp_pq}, the ODE for $\Psi_i(t)$ is of the form
			\begin{equation}
				\label{eq:lambda_some}
				\frac{d\Psi_i}{dt}=\Big(\gamma e^{-\theta t}+\Psi_i\Big)\big(A+B(t)\big),
			\end{equation}
			where $A>0$ and $B(t)\to 0$ as $t\to\infty$. If $T=\infty$, the numerical solution for $\Psi(t)$ will grow exponentially as $t\to\infty$, while the analytical solution should decay to $0$. To see this, we first prove
			\begin{lemma}
				\label{lem:lambda_asymptotic}
				Let $\Psi_i$ be the solution of~\eqref{eq:lambda_some}. Then 
				$
				\Psi_i(t)\sim c_2 e^{-\theta t}$ as $ t\to\infty$, where $c_2:= -\frac{\gamma A}{\theta+A}$.
			\end{lemma}
			\begin{proof}
				Since $\lim\limits_{t\to\infty}B=0$, then
				\begin{equation}
					\label{eq:lam_sol1}
					\frac{d\Psi_i}{dt}\sim \Big(\gamma e^{-\theta t}+\Psi_i\Big)A, \qquad t\to\infty.
				\end{equation}
				Solving this ODE gives
				\begin{equation}
					\label{eq:lam_sol}
					\Psi_i \sim c_1 e^{A t}+c_2 e^{-\theta t}.
				\end{equation}
				Since $\lim\limits_{t\to\infty}\Psi_i=0$, then $c_1=0$, and so $\Psi_i\sim c_2 e^{-\theta t}$.
				Substituting this back into~\eqref{eq:lam_sol1} gives~$c_2$.
			\end{proof}
			
			Equation~\eqref{eq:lam_sol} explains the spurious exponential growth of the numerical solution. Analytically, $c_1=0$, but numerically we only have that $c_1 \approx 0$, and so $c_1 e^{A_i t}$ grows exponentially for $t\gg 1$. To avoid this, we modify the shooting method as follows:
			\begin{enumerate}
				\item 
				Find the first time $t^*$ at which $1-f^0(t):=f(t;{\bf s}(t)\equiv 0)$ is smaller than a certain tolerance. If no such time exists, let $t^*=T$.
				\item 
				Solve the boundary-value problem for  $0\leq t\leq t^*$ using the shooting method, with the terminal boundary condition being
				$\Psi_i(t^*)=c_2 e^{-\theta t^*}$.
				\item
				Solve the boundary-value problem for $t^*\leq t<\infty$ while replacing $\Psi_i(t)$ with $c_2 e^{-\theta t}$. 
				\item
				Use the numerical solution of the boundary-value problem to compute ${\bf s}(t)$.
			\end{enumerate}
			
			Although the previous analysis was for $T=\infty$, the spurious exponential growth of the numerical solution also occurs for $T\gg 1$, so we use the same method.

			\section*{Acknowledgments}
			\label{sec:acknowledgments}
			We thank Gila Fruchter and Christophe Van den Bulte for useful discussions.


\begin{thebibliography}{10}
	
	\bibitem{Akbarpour2018}
	M.~Akbarpour, S.~Malladi, and A.~Saberi.
	\newblock Just a few seeds more: Value of network information for diffusion.
	\newblock {\em American Economic Review}, 115:3713–48, 2025.
	
	\bibitem{Albert-00}
	R.~Albert, H.~Jeong, and A.L. Barab\'asi.
	\newblock Error and attack tolerance of complex networks.
	\newblock {\em Nature}, 406:378--382, 2000.
	
	\bibitem{Anderson-92}
	R.M. Anderson and R.M. May.
	\newblock {\em Infectious Diseases of Humans}.
	\newblock Oxford University Press, Oxford, 1992.
	
	\bibitem{Balderrama2022-qa}
	R.~Balderrama, J.~Peressutti, J.~P. Pinasco, F.~Vazquez, and C.~S. de~la Vega.
	\newblock Optimal control for a {SIR} epidemic model with limited quarantine.
	\newblock {\em Scientific Reports}, 12:12583, 2022.
	
	\bibitem{Banerjee-13}
	A.~Banerjee, A.G. Chandrasekhar, E.~Duflo, and M.O. Jackson.
	\newblock The diffusion of microfinance.
	\newblock {\em Science}, 341:1236498, 2013.
	
	\bibitem{Bass-69}
	F.M. Bass.
	\newblock A new product growth model for consumer durables.
	\newblock {\em Management Sci.}, 15:1215--1227, 1969.
	
	\bibitem{Bass1994}
	F.M. Bass, T.V. Krishnan, and D.C. Jain.
	\newblock Why the {B}ass model fits without decision variables.
	\newblock {\em Marketing Science}, 13:203--223, 1994.
	
	\bibitem{Beaman2021}
	L.~Beaman, A.~Ben~Yishay, J.~Magruder, and A.~M. Mobarak.
	\newblock Can network theory-based targeting increase technology adoption?
	\newblock {\em American Economic Review}, 111:1918--43, 2021.
	
	\bibitem{Dockner-1988}
	E.~Dockner and S.~Jørgensen.
	\newblock Optimal advertising policies for diffusion models of new product
	innovation in monopolistic situations.
	\newblock {\em Management Science}, 34:119--130, 1988.
	
	\bibitem{OR-10}
	G.~Fibich and R.~Gibori.
	\newblock Aggregate diffusion dynamics in agent-based models with a spatial
	structure.
	\newblock {\em Oper. Res.}, 58:1450--1468, 2010.
	
	\bibitem{fibich2021diffusion}
	G.~Fibich and A.~Golan.
	\newblock Diffusion of new products with heterogeneous consumers.
	\newblock {\em Mathematics of Operations Research}, 48:257--287, 2023.
	
	\bibitem{Fibichcompartmental}
	G.~Fibich, A.~Golan, and S.~Schochet.
	\newblock Compartmental limit of discrete {B}ass models on networks.
	\newblock {\em Discrete and Continuous Dynamical Systems - B}, 28:3052--3078,
	2023.
	
	\bibitem{Bass-monotone-convergence-23}
	G.~Fibich, A.~Golan, and S.~Schochet.
	\newblock Monotone convergence of discrete {B}ass models.
	\newblock {\em Operations Research Letters}, 64:107363, 2026.
	
	\bibitem{Fruchter-22}
	G.E. Fruchter, A.~Prasad, and C.~Van~den Bulte.
	\newblock Too popular, too fast: Optimal advertising and entry timing in
	markets with peer influence.
	\newblock {\em Management Science}, 68:4725--4741, 2022.
	
	\bibitem{Fruchter2011}
	G.E. Fruchter and C.~{Van den Bulte}.
	\newblock Why the generalized {B}ass model leads to odd optimal advertising
	policies.
	\newblock {\em International Journal of Research in Marketing}, 28:218--230,
	2011.
	
	\bibitem{Goldenberg2001}
	J.~Goldenberg, B.~Libai, and E.~Muller.
	\newblock Talk of the network: A complex systems look at the underlying process
	of word-of-mouth.
	\newblock {\em Marketing Letters}, 12:211--223, 2001.
	
	\bibitem{Hopp-04}
	W.J. Hopp.
	\newblock Ten most influential papers of {M}anagement {S}cience's first fifty
	years.
	\newblock {\em Management Sci.}, 50:1763--1893, 2004.
	
	\bibitem{Horsky-83}
	D.~Horsky and L.S. Simon.
	\newblock Advertising and the diffusion of new products.
	\newblock {\em Marketing Science}, 2:1--17, 1983.
	
	\bibitem{Jackson-08}
	M.O. Jackson.
	\newblock {\em Social and Economic Networks}.
	\newblock Princeton University Press, Princeton and Oxford, 2008.
	
	\bibitem{kamienschwartz}
	M.I. Kamien and N.L. Schwartz.
	\newblock {\em Dynamic Optimization: The Calculus of Variations and Optimal
		Control in Economics and Management}.
	\newblock Elsevier Science, second edition, 1991.
	
	\bibitem{Krishnan2006}
	T.V. Krishnan and D.C. Jain.
	\newblock Optimal dynamic advertising policy for new products.
	\newblock {\em Management Science}, 52:1957--1969, 2006.
	
	\bibitem{Libai-2013}
	B.~Libai, E.~Muller, and R.~Peres.
	\newblock Decomposing the value of word-of-mouth seeding programs: Acceleration
	versus expansion.
	\newblock {\em Journal of Marketing Research}, 50:161--176, 2013.
	
	\bibitem{Muller2019}
	E.~Muller and R.~Peres.
	\newblock The effect of social networks structure on innovation performance: A
	review and directions for research.
	\newblock {\em International Journal of Research in Marketing}, 36:3--19, 2019.
	
	\bibitem{Niu-02}
	S.C. Niu.
	\newblock A stochastic formulation of the {B}ass model of new product
	diffusion.
	\newblock {\em Math. Problems Engrg.}, 8:249--263, 2002.
	
	\bibitem{Pastor-Satorras-01}
	R.~Pastor-Satorras and A.~Vespignani.
	\newblock Epidemic spreading in scale-free networks.
	\newblock {\em Phys. Rev. Lett.}, 86:3200--3203, 2001.
	
	\bibitem{Rossman2021}
	G.~Rossman and J.C. Fisher.
	\newblock Network hubs cease to be influential in the presence of low levels of
	advertising.
	\newblock {\em Proceedings of the National Academy of Sciences},
	118:e2013391118, 2021.
	
	\bibitem{Schmitt2011}
	P.~Schmitt, B.~Skiera, and C.~Van~den Bulte.
	\newblock Referral programs and customer value.
	\newblock {\em Journal of Marketing}, 75:46--59, 2011.
	
	\bibitem{Strang-98}
	D.~Strang and S.A. Soule.
	\newblock Diffusion in organizations and social movements: From hybrid corn to
	poison pills.
	\newblock {\em Annu. Rev. Sociol.}, 24:265--290, 1998.
	
	\bibitem{Teng-1983}
	J.-T. Teng and G.L. Thompson.
	\newblock Oligopoly models for optimal advertising when production costs obey a
	learning curve.
	\newblock {\em Management Science}, 29:1087--1101, 1983.
	
	\bibitem{Teng-1984}
	G.L. Thompson and J.-T. Teng.
	\newblock Optimal pricing and advertising policies for new product oligopoly
	models.
	\newblock {\em Marketing Science}, 3:148--168, 1984.
	
	\bibitem{VanDenBulte2018}
	C.~Van Den~Bulte, E.~Bayer, B.~Skiera, and P.~Schmitt.
	\newblock How customer referral programs turn social capital into economic
	capital.
	\newblock {\em Journal of Marketing Research}, 55:132--146, 2018.
	
\end{thebibliography}
		\end{document}